\documentclass[11pt,final]{amsart}
\usepackage{amssymb, amsmath, amsthm}
\usepackage{mathrsfs}  
\usepackage{graphicx}
\usepackage[colorlinks=true, citecolor=blue, linkcolor=blue, urlcolor=blue]{hyperref}
 \usepackage{xcolor}
 \usepackage{tcolorbox}
\usepackage{mathrsfs}  

\usepackage{bm}

\usepackage{ifdraft}
\ifoptionfinal{
\usepackage[disable]{todonotes}
}{
\usepackage[norefs, nocites]{refcheck}

\usepackage[notref, notcite]{showkeys}
\usepackage[bordercolor=white, color=white]{todonotes}
}

\makeatletter\providecommand\@dotsep{5}\def\listtodoname{List of Todos}\def\listoftodos{\hypersetup{linkcolor=black}\@starttoc{tdo}\listtodoname\hypersetup{linkcolor=blue}}\makeatother

\usepackage{hyperref}
\usepackage{etoolbox}

\newtheorem{theorem}{Theorem}[section]
\newtheorem{lemma}[theorem]{Lemma}
\newtheorem{corollary}[theorem]{Corollary}
\newtheorem{proposition}[theorem]{Proposition}
\newtheorem{definition}[theorem]{Definition}
\newtheorem{hypothesis}[theorem]{Hypothesis}
\theoremstyle{remark}
\newtheorem{remark}{Remark}

\numberwithin{equation}{section}

\def\R{\mathbb R}

\def\N{\mathbb N}
\def\g{\textsl g}

\renewcommand{\leq}{\leqslant}
\renewcommand{\geq}{\geqslant}

\def\p{\partial}

\newcommand*\xbar[1]{%
   \hbox{%
     \vbox{%
       \hrule height 0.5pt 
       \kern0.5ex
       \hbox{%
         \ensuremath{#1}%
       }%
     }%
   }%
} 

\title[Inverse problem for isotropic nonautonomous heat flows]{An inverse boundary value problem for isotropic nonautonomous heat flows}

\author[Ali Feizmohammadi]{Ali Feizmohammadi}
\address{Fields institute, 222 College St, Toronto, ON M5T 3J1}
\email{afeizmoh@fields.utoronto.ca}


\begin{document}


\maketitle
\begin{abstract}
We study an inverse boundary value problem on the determination of principal order coefficients in isotropic nonautonomous heat flows stated as follows; given a medium, and in the absence of heat sources and sinks, can the time-dependent thermal conductivity and volumetric heat capacity of the medium be uniquely determined from the Cauchy data of temperature and heat flux measurements on its boundary? We prove uniqueness in all dimensions under an assumption on the thermal diffusivity of the medium, which is defined as the ratio of the thermal conductivity and volumetric heat capacity. As a corollary of our result for isotropic media, we also obtain a uniqueness result, up to a natural gauge, in two-dimensional anisotropic media. Our assumption on the thermal diffusivity is related to construction of certain families of exponential solutions to the heat equation.  
\end{abstract}

\section{Introduction}

Let $T>0$ and let $\Omega\subset \R^n$, $n\geq 2$, be a bounded domain with a smooth  boundary. Let $\rho, \sigma$ be smooth positive functions on $M=[0,T]\times \overline\Omega$ where $\overline\Omega$ is the closure of $\Omega$ in $\R^n$. In the absence of heat sources and sinks, the heat flow in an isotropic medium with a nonuniform density is governed by the first law of thermodynamics and satisfies the equation
\begin{equation}\label{heat}
\begin{aligned}
\begin{cases}
\rho(t,x)\,\p_t u -\nabla\cdot(\sigma(t,x)\nabla u)=0 
&\text{on  $M^{\textrm{int}}$},
\\
u= f & \text{on $\Sigma=(0,T)\times \p \Omega$},
\\
u|_{t=0}= 0 & \text{on $\Omega$}.
\end{cases}
\end{aligned}
\end{equation}
Here, $\nabla$ and $\cdot$ respectively denote the gradient operator and the Euclidean real inner product on $\R^n$,  $u(t,x)$ is the temperature at time $t$ and position $x$, $\rho(t,x)$ denotes the volumetric heat capacity of the medium while $\sigma(t,x)$ denotes its thermal conductivity. Physically, $\rho(t,x)$ is the amount of heat energy that must be added to one unit of volume in the medium at time $t\in [0,T]$ and position $x\in \Omega$ in order to cause an increase of one unit in its temperature. On the other hand, $\sigma(t,x)$ measures the transport of heat energy at time $t\in [0,T]$ and position $x\in \Omega$ due to random molecular motion across a temperature gradient. We mention that the ratio $\frac{\rho(t,x)}{\sigma(t,x)}$ is the thermal diffusivity of the medium which measures its ability to conduct thermal energy relative to its ability to store thermal energy, see e.g. \cite{HS}.

Throughout this paper, we use the standard mixed Sobolev spaces
\begin{equation}\label{Sobolev_def_1}
H^{r,s}(M):= H^r(0,T;L^2(\Omega)) \cap L^2(0,T;H^s(\Omega)),\end{equation}
and 
\begin{equation}
\label{Sobolev_def_2}
H^{r,s}(\Sigma):= H^r(0,T;L^2(\p\Omega)) \cap L^2(0,T;H^s(\p\Omega)),\end{equation}
for any $r,s\geq 0$. It is classical (see e.g. \cite[Section 2]{LM}) that given any Dirichlet boundary data $f$ in the space
$$
\mathcal H(\Sigma)=\{h\in H^{\frac{3}{4},\frac{3}{2}}(\Sigma)\,:\, h|_{t=0}=0\},
$$ 
the initial boundary value problem \eqref{heat} admits a unique solution 
\begin{equation}
\label{energy_space}
u \in H^{1,2}(M) \quad \text{with}\quad \p_\nu u|_{\Sigma}\in H^{\frac{1}{4},\frac{1}{2}}(\Sigma),
\end{equation}
where $\p_\nu u|_{\Sigma}$ is the function $\nabla u(t,x)\cdot \nu(x)$ with $(t,x)\in \Sigma$ and $\nu=\nu(x)$ is the exterior unit normal vector field on $\p \Omega$. Moreover, the dependence of \eqref{energy_space} on $f\in \mathcal H(\Sigma)$ is continuous. We define the Dirichlet-to-Neumann map associated to \eqref{heat} via the mapping
\begin{equation}\label{DN_map}
\Lambda_{\rho,\sigma}:f\mapsto (\sigma(t,x)\p_{\nu}u(t,x))|_{\Sigma},\qquad \forall\, f\in \mathcal H(\Sigma),
\end{equation}
where $u$ is the unique solution to \eqref{heat} with Dirichlet boundary data $f$. Physically, $\Lambda_{\rho,\sigma}(f)$ measures the induced heat flux that exits through the boundary. This paper is concerned with the question of unique determination of a smooth positive $\rho$ and $\sigma$, given the knowledge of the Dirichlet-to-Neumann map $\Lambda_{\rho,\sigma}$. In other words, we ask
\begin{tcolorbox}[width=\textwidth,colback={white},title={},colbacktitle=yellow,coltitle=blue]      
\begin{itemize}
 \item[(IP)]{{\em Is the mapping $(\rho,\sigma)\mapsto \Lambda_{\rho,\sigma}$ injective?}}
\end{itemize}
\end{tcolorbox}  
\noindent To further illustrate the physical motivation behind (IP), let us mention that spacetime dependent coefficients in isotropic parabolic equations are also encountered in the study of inverse problems for nonlinear parabolic equations modeling physical processes, see e.g. \cite{Frank,RSSP}. Indeed, through a first order linearization procedure introduced by Isakov in \cite{Isakov1}, it is possible to reduce the determination of time-independent coefficients appearing in non-linear equations to the recovery of spacetime dependent coefficients appearing
in a linear equation, see also the recent result \cite{FKU} for further examples of linearizations of inverse problems for nonlinear parabolic equations. 

Physical motivation aside, (IP) may be viewed as an evolutionary analogue of the well known isotropic Calder\'{o}n problem, also known as electrical impedance tomography, that is concerned with recovery of a static isotropic electrical conductivity through making voltage and current measurements on the surface of a medium \cite{Calderon}. The seminal works of Sylvester and Uhlmann in \cite{SU} and Nachman in \cite{Nachman} solve the Calder\'{o}n problem in dimensions $n\geq3$ and $n=2$ respectively, see also \cite{Uhl} for review and \cite{DKSU,FKLS,LU} for state of the art results on the anisotropic formulation of the Calder\'{o}n problem.  

\subsection{Previous literature} The previous literature of results for the inverse problem (IP) is rather sparse. This is not an artifact of heat equation per se, as recovery of spacetime dependent principal order coefficients from boundary data is a very challenging question across many evolutionary PDEs such as the wave equation or the dynamical Schr\"{o}dinger equation. For example, the analogous inverse problem for the determination of a wave speed $c(t,x)$ from Cauchy boundary data $(u|_{\Sigma},c\,\p_\nu u|_{\Sigma})$ of smooth solutions to the isotropic wave equation
\begin{equation}\label{wave} 
\p^2_{tt}u -c^2(t,x)\, \Delta u=0,\qquad \text{on $(0,T)\times \Omega$},\end{equation}
subject to $u|_{t\leq 0}=0$ is also open, even when $c(t,x)$ is assumed to be near a constant. Indeed, the latter inverse problem has been solved only when the wave speed is time-independent \cite{Bel87} or more generally if it depends real analytically on the time variable \cite{eskin}. We remark that in the case of \eqref{wave} and owing to the principle of propagation of singularities, it is possible to uniquely reconstruct certain nonlocal boundary scattering maps involving the wave speed, see for example the work of Stefanov and Yang \cite{SY} in the broader context of wave equations on Lorentzian manifolds. However, the inversion of the scattering map for spacetime dependent wave speeds is open. In line with this, Oksanen, Salo, Stefanov and Uhlmann \cite{OSSU} have recently studied inverse problems for broad classes of real principal type operators with nontrivial bicharacteristic flows, and showed that it is possible to reduce the question of recovering principal coefficients in such equations to an open problem regarding injectivity of certain scattering maps related to the coefficients. We remark that \eqref{heat} is not real principal type. Moreover, owing to parabolic regularity, approaches similar to \cite{OSSU,SY} will not apply in the case of (IP).   

Returning to the previous literature on (IP), when the coefficients $\rho=\rho(t)$ and $\sigma=\sigma(t)$ are functions of time only, injectivity results are available, see for example the work of Cannon in \cite{Ca}. When $\rho=\rho(x)$ and $\sigma=\sigma(x)$ are both assumed to be time-independent, (IP) follows from the work of Canuto and Kavian in \cite{CK}, see also the earlier work of Isakov in \cite{Isa}. We refer the reader to the work of Katchalov, Kurylev, Lassas and Mandache in \cite{KKLM} that considers broad classes of evolution equations with time-independent coefficients as well as the work of Kian, Li, Liu and Yamamoto in \cite{KLLY} that considers the problem of recovering time-independent coefficients in fractional order parabolic equations from one measurement. We emphasize that the latter works \cite{Isa,CK,KKLM,KLLY} all fundamentally rely on the time-independence of the coefficients in order to reduce the inverse coefficient determination problem to the Gel'fand inverse spectral problem that was studied in \cite{BK92}. Before closing, let us mention that when $\rho\equiv \sigma\equiv1$, there are uniqueness results in the literature for recovery of spacetime dependent {\em lower} order coefficients in such equations, that is to say, equations of the form
$$\p_t u -\Delta u + B(t,x)\cdot \nabla u +q(t,x) u=0.$$
For results in the latter direction, we refer the reader to the works \cite{CaK,ChK,Isa92} and the references therein. 

\subsection{Main result} In this paper, we give the first resolution of (IP) in the direction of recovering a spacetime dependent pair of principal order coefficients $\sigma$ and $\rho$ under a suitable assumption on the topology of the domain $\Omega$ as well as an assumption on the thermal diffusivity of the medium. 
\begin{hypothesis}
\label{hypo}
Defining the time-dependent family of conformally Euclidean Riemannian metrics 
$$\g(t,x)= \frac{\rho(t,x)}{\sigma(t,x)}\left((dx_1)^2+\ldots+(dx_n)^2\right), \quad t\in [0,T],\quad x\in \overline\Omega,$$
the Riemannian manifold $(\overline\Omega,\g(t,\cdot))$ is simple for any $t\in [0,T]$. (By simple, we mean that $\Omega$ is simply connected, that $\p \Omega$ is strictly convex with respect to $\g(t,\cdot)$ and finally that $(\overline\Omega,\g(t,\cdot))$ does not have any conjugate points.)
\end{hypothesis}
Let us give two examples satisfying Hypothesis~\ref{hypo}. In both examples we assume that $\Omega$ is any simply connected domain with a smooth strictly convex boundary. As a first example, let $l\in C^0([0,T])$ be a continuous positive function and assume that the diffusivity function $\gamma(t,x):=\frac{\rho(t,x)}{\sigma(t,x)}$ is sufficiently close to $l(t)$ for all $(t,x)\in M$, namely
$$\|\gamma-l\|_{C^0(0,T;C^2(\overline\Omega))}\leq \varepsilon,$$
for some $\varepsilon>0$ sufficiently small. It is clear that this example satisfies Hypothesis~\ref{hypo}. To produce another example, we first recall that simply connected Riemannian manifolds of negative curvature do not have conjugate points. Using the law of changes of the curvature tensor and changes of second fundamental form under conformal scalings of the metric, we deduce that Hypothesis~\ref{hypo} is also satisfied if the diffusivity function $\gamma$ satisfies $\p_\nu \gamma|_{\Sigma}\geq 0$ and additionally there holds
$$ \sum_{j,k=1}^n\frac{\p^2\gamma^{-\frac{1}{2}}}{\p x_j\p x_k}(t,x)\,X^j\,X^k \leq 0,\quad \forall \,(t,x)\in M,\quad \forall\,X\in \R^n.$$
Our main result can be stated as follows.
\begin{theorem}
	\label{thm_main}
	Let $T>0$ and let $\Omega\subset \R^n$, $n\geq 2$, be a bounded domain with a smooth boundary. For $j=1,2$, let $\rho_j(t,x)$ and $\sigma_j(t,x)$ be smooth positive functions on $M=[0,T]\times\overline\Omega$. Suppose that Hypothesis~\ref{hypo} is satisfied with $\rho=\rho_j$ and $\sigma=\sigma_j$, $j=1,2$. If,
	$$\Lambda_{\rho_1,\sigma_1}(f)=\Lambda_{\rho_2,\sigma_2}(f)\qquad \forall\, f\in \mathcal H(\Sigma),$$
	then 
	$$\rho_1=\rho_2\quad \text{and}\quad \sigma_1=\sigma_2\quad \text{on $M$}.$$
\end{theorem}
We will in fact prove a stronger version of the theorem that is local in time in the following sense; in order to derive $(\rho_1,\sigma_1)=(\rho_2,\sigma_2)$ at a fixed time $t=t_0$, it suffices to consider Dirichlet boundary data $f$ that is supported in any neighborhood of the time slice $t=t_0$ on $\Sigma$ and also that the Dirichlet-to-Neumann maps $\Lambda_{\rho_j,\sigma_j}(f)$, $j=1,2$ are identical in that same neighborhood and also that Hypothesis~\ref{hypo} is satisfied in a neighborhood of $t=t_0$. 

\subsection{Result on anisotropic conductivities and future work}
\label{sec_anisotropic}
The inverse problem (IP) is physically motivated by nonautonomous heat conduction in isotropic media. An analogous problem may be posed for a medium with an anisotropic thermal conductivity denoted by a symmetric positive matrix $(a^{jk}(t,x))_{j,k=1}^n$. In this model one studies the more general initial boundary value problem 
\begin{equation}\label{physical_heat}
\begin{aligned}
\begin{cases}
\rho(t,x)\,\p_t u-\sum_{j,k=1}^n\frac{\p}{\p x^j}\left(a^{jk}(t,x)\frac{\p u}{\p x^k}\right)=0, 
&\text{on  $(0,T)\times \Omega$},
\\
u= f & \text{on $\Sigma$},
\\
u|_{t=0}= 0 & \text{on $\Omega$}.
\end{cases}
\end{aligned}
\end{equation}
As in the case of (IP), consider the Cauchy boundary data of temperature and heat flux measurements on $\Sigma$, namely,
\begin{equation}\label{heat_flux}
\mathscr C_{\rho,a}=\{(f,\sum_{j,k=1}^na^{jk}(t,x)\frac{\p u}{\p x^j}(t,x)\, \nu_k(x)\big|_{(t,x)\in \Sigma})\,:\, \text{$u$ solves \eqref{physical_heat}}\},\end{equation} 
where $\nu=(\nu_1,\ldots,\nu_n)$ is the exterior unit normal vector field on the boundary $\p\Omega$. The following generalization of (IP) can then be posed; given the knowledge of the Cauchy boundary data $\mathscr C_{\rho,a}$, is it possible to determine the unknown matrix conductivity $(a^{jk}(t,x))_{j,k=1}^n$ and the volumetric heat capacity $\rho(t,x)$ on $[0,T]\times \Omega$?  

A natural obstruction to uniqueness arises in the form of changes of coordinates. Precisely, if $F_t:\overline\Omega \to \overline\Omega$, $t\in [0,T]$ is a family of diffeomorphisms that depend smoothly on $t\in [0,T]$ and such that $F_t|_{\p \Omega}$, $t\in [0,T]$, is the identity map, then $\mathscr C_{\rho_1,a_1}=\mathscr C_{\rho_2,a_2}$, where
\begin{equation}\label{gauge_coord} \rho_1(t,x)=\rho_2(t,F_t(x))\quad \text{and} \quad a_1(t,x)=DF_t(x)^\textrm{T} a_2(t,F_t(x)) DF_t(x).\end{equation}
Here $DF_t$ is the Jacobian matrix of $F_t$ and $DF_t^\textrm{T}$ is its transpose. Thus, the best one can hope for is to recover the scalar $\rho$ and the matrix $a$ up to the gauge \eqref{gauge_coord}. 

When $n=2$, and by using isothermal coordinates at each time $t\in [0,T]$, it is possible to reduce \eqref{physical_heat} to the isotropic model \eqref{heat}, see e.g. \cite{Sy1} for the details of this reduction. Thus, our main theorem has the following immediate corollary. 

\begin{corollary}
	\label{cor1}
	Let $T>0$ and let $\Omega\subset \R^2$ be a simply connected bounded domain with a smooth boundary. For $j=1,2,$ let $\rho_j\in C^{\infty}([0,T]\times\overline \Omega)$ be positive and let $a_j(t,x)=(a_j^{kl}(t,x))_{k,l=1,2}$ be a smooth positive symmetric matrix on $M=[0,T]\times \overline\Omega$. For $j=1,2$, denote by $a_j^{-1}$ the inverse of the matrix $a_j$ and suppose that the families of Riemannian manifolds  $(\overline \Omega, \rho(t,\cdot)\, a_j^{-1}(t,\cdot))$, $t\in [0,T]$ all have a strictly convex boundary and have no conjugate points. If,
	$$\mathscr C_{\rho_1,a_1}=\mathscr C_{\rho_2,a_2},$$
	then, there exists a family of diffeomorphisms $F_t:\overline\Omega\to \overline\Omega$ smoothly depending on $t\in [0,T]$ that fix the boundary, such that \eqref{gauge_coord} holds.
\end{corollary}

We will leave the study of inverse problems for nonautonomous heat flows in anisotropic media of dimension $n\geq 3$ to a future work.

\subsection{Outline of key ideas}
In order to highlight the main ideas in this work clearly, let us first briefly recall the (isotropic) Calder\'{o}n problem that is concerned with recovery of the electrical conductivity $\sigma_{\textrm{el}}(x)$ of a medium $\Omega$ from the Cauchy data set $(u|_{\p\Omega},\sigma_{\textrm{el}}\p_\nu u|_{\p\Omega})$ of solutions to the electrical conduction equation
$$ \nabla\cdot(\sigma_{\textrm{el}}(x)\nabla u)=0\qquad \text{on $\Omega$}.$$
In the case of the Calder\'{o}n problem and its resolution in \cite{SU}, one of the starting points is to use a change of variables to reduce the Calder\'{o}n problem to an inverse problem for the operator $-\Delta+q(x)$, where the unknown electrical conductivity $\sigma_{\textrm{el}}$ reappears as a lower order coefficient $q$. As the Dirichlet-to-Neumann map associated to the latter equation is self adjoint, the problem of unique determination of $q$ reduces to a claim about density of products of two solutions to equations of the form $-\Delta u_j +q_ju_j=0$ on $\Omega$, with $j=1,2$; namely that if 
\begin{equation}\label{sch_iden}
\int_\Omega (q_1-q_2)u_1u_2\,dx=0,\end{equation}
for all solutions $u_1$ and $u_2$, then $q_1=q_2$. Sylvester and Uhlmann achieved this density property by introducing a special family of solutions to the latter equations that are commonly known as {\em Complex Geometric Optic} solutions (CGOs). These are solutions of the form 
\begin{equation}
\label{CGO}
u_{1}(x)=e^{\tau \Phi(x)}(1+r_1(x))\quad \text{and}\quad u_{2}(x)=e^{-\tau \Phi(x)}(e^{\textrm{i}\xi\cdot x}+r_2(x)),
\end{equation}
where $\tau$ is a large parameter, $\xi\in \R^n$, $\Phi$ is an explicit complex valued linear function that is independent of $q_j$, $j=1,2$ and finally $r_j$ is a remainder term that decays proportionally to $\frac{1}{|\tau|}$. We emphasize that the reduction of the Calder\'{o}n problem to an inverse problem for $-\Delta+q$ is rather significant as it will allow one to cancel the exponential weights in the product  $u_1u_2$ and thus conclude the density property by taking $\tau \to \infty$. Let us also remark that the construction of the phase function $\Phi$ is rather rigid. This is due to the obstructions imposed by deriving limiting Carleman estimates for the conjugated Laplace operator, see for example \cite{DKSU} in the broader context of Laplace operator on Riemannian manifolds.
 
Unlike the Calder\'{o}n problem, there is no change of variables that reduces (IP) to an inverse problem for determination of lower order coefficients only. Indeed, the best one can achieve via changes of variables is to reduce equation \eqref{heat} to an equation of the form \eqref{heat_alt} where a principal order coefficient $\gamma=\frac{\rho}{\sigma}$ remains present in the equation as well as the introduction of a new lower order coefficient $q$. As we will see in Sections~\ref{sec_thm}--\ref{sec_thm_q}, recovery of the coefficient $\gamma$ will be intricately tied to the question of boundary rigidity of conformally Euclidean manifolds while recovery of $q$ is tied to injectivity of geodesic ray transforms. This reformulation of (IP) (as well as an extension of the coefficients that will be needed later) requires boundary determination of full Taylor series of $\sigma$ and $\rho$ on $\Sigma$. This is a step that is also carried out in the case of the Calder\'{o}n problem, see for example \cite{KV,LU,SU}. More generally, boundary determinations for inverse problems associated to real principal type operators has recently been studied in \cite{OSSU}. We remark that our boundary determination result, namely Proposition~\ref{prop_boundary}, is not covered by any of the previous literature as the heat operator is not real principal type. We hope that the boundary determination approach taken here may be of independent future interest.

As a next step, we proceed to introduce a family of {\em exponentially growing/decaying} solutions to equation \eqref{heat_alt} that will loosely play the role of the CGO solutions in the context of the Calder\'{o}n problem. The construction of the exponential solutions here will be fundamentally different and in particular choosing a phase function $\Phi$ that is the same as \eqref{CGO} will fail here as we will not be able to obtain remainder estimates with proper decay rates in $\tau$. We modify our construction of exponential solutions to take into account the parabolic nature of \eqref{heat_alt} as well as the vanishing initial condition. 

There are three key new ideas here. First, as we will see in Section~\ref{sec_carleman} and Section~\ref{sec_exp}, the construction of the phase functions in our paper is less rigid in comparison with the Calder\'{o}n problem. This is related to the derivation of limiting Carleman estimates that enforce the decay of the remainder terms in the exponential solutions. Here, by limiting Carleman estimate we roughly mean an energy estimate for conjugated heat operators with exponential weights, where the weight function has to additionally satisfy a suitable eikonal equation.  Surprisingly, limiting Carleman estimates for the conjugated heat operator can always be derived through natural energy estimates, see Proposition~\ref{energy_prop}. In particular, this step of the analysis in parabolic equations does not impose that the principal coefficients of the equation are independent of one of the variables in the equation, which is an obstruction that appears in the study of limiting Carleman estimates for elliptic equations \cite{DKSU}. This feature of parabolic equations also suggests that the generalization of our work to the setup of anisotropic conductivities (see Section~\ref{sec_anisotropic}) in dimensions $n\geq 3$ should be possible, under an analogous assumption as in Hypothesis~\ref{hypo}. We remark that when $\rho\equiv \sigma\equiv 1$ a similar Carleman estimate has appeared in the literature with a linear phase function, see e.g. \cite{ChK}. However, the phase function in this paper is nonlinear.

Second, as equation \eqref{heat_alt} is not self-adjoint, we will not have a direct analogue of identity \eqref{sch_iden}. Nevertheless, it is possible to obtain boundary integral identities that play an analogous role to \eqref{sch_iden} and in which the product of two exponentially growing solutions can be used. Finally, we emphasize that equation \eqref{heat_alt} has an unknown principal order coefficient, and therefore the corresponding phase functions in the construction of our special solutions will be a priori unknown as well. This is in fact the major challenge in solving (IP) compared to the Calde\'r{o}n problem as the cancellation of the exponential weights in the products of our special solutions to \eqref{heat_alt} is not possible. We will overcome this technical issue via a purely geometric lemma (Lemma~\ref{lem_geo}) in Section~\ref{sec_thm}. This is also the reason that our proof of the main theorem in Section~\ref{sec_thm} is via a proof by contradiction and a maximality argument.  

\subsection*{Acknowledgments} 
This work is supported by the Fields institute for research in mathematical sciences.  

\section{Determination of the Taylor series of coefficients on $\Sigma$}
\label{sec_boundary}

This section is concerned with the proof of the following proposition.

\begin{proposition}
	\label{prop_boundary}
	Let the hypotheses of Theorem~\ref{thm_main} be satisfied. Then,
	$$\p_\nu^{k} \sigma_1|_{\Sigma}= \p_\nu^{k} \sigma_2|_{\Sigma}\quad \text{and}\quad \p_\nu^{k} \rho_1|_{\Sigma}= \p_\nu^{k} \rho_2|_{\Sigma}$$ 
	for any $k\in \{0\}\cup \N$. (Here, to make sense of the derivatives, one should locally extend $\nu$ smoothly in a neighborhood of the boundary.)
\end{proposition}

\subsection{Oscillatory localized solutions near boundary}

We begin by considering a local normal coordinate system $(y_1,y')=(y_1,\ldots,y_n)$ in a neighborhood of the boundary $\p \Omega$, namely 
\begin{equation}\label{epsilon_def}
 Y: (-\epsilon,\epsilon)\times \p\Omega \to \R^n,
 \end{equation} 
that is defined for sufficiently small $\epsilon$ via 
\begin{equation}\label{fermi} 
Y(y_1,y')= y'+y_1\,\nu(y'),\end{equation}
where the previous sum is a vector summation in $\R^n$ and $y'\in \p\Omega$ is to be viewed as a vector in $\R^n$. Note that $y_1<0$ inside $\Omega$. It is well known that the pull back of the Euclidean metric $\mathbb E^n$ with respect to this local diffeomorphism takes the form
\begin{equation}\label{metric_form}
g(y_1,y')=Y^\star \mathbb E^n=(dy_1)^2+g'(y_1,y'),\end{equation}
where 
$$g'(y_1,y')=\sum_{\alpha,\beta=2}^{n}g'_{\alpha\beta}(y_1,y')\,dy_\alpha\,dy_\beta$$ 
are a family of smooth Riemannian metrics on $\p \Omega$ that depend smoothly on $y_1\in (-\epsilon,\epsilon)$. We will denote by $\langle\cdot,\cdot\rangle_g$, $\nabla^g$ and $\Delta_g$, the inner product on $\R^n$, gradient operator and Laplace--Beltrami operator written with respect to the $y$-coordinate system. Writing $g^{jk}$ for the elements of the inverse matrix of $g=(g_{jk})_{j,k=1}^n$, the latter two operators are respectively defined via the expressions
$$\nabla^g f = \sum_{j,k=1}^n g^{jk}\,\frac{\p f}{\p y_j}\frac{\p}{\p y_k},$$
and 
$$ \Delta_g f = \frac{1}{\sqrt{\det g}}\sum_{j,k=1}^n\frac{\p}{\p y_j}\left(\sqrt{\det g}\,g^{jk}\frac{\p f}{\p y_k}\right).$$

Let us fix $t_0\in (0,T)$ and a point $p_0\in \p\Omega$. Without loss of generality we may assume that $p_0$ is the origin in the $y$-coordinate system. Our aim is to construct an approximate solution to equation \eqref{heat} that is supported near the point $(t_0,p_0)\in M$ and is highly oscillatory with exponential decay inside $M$ away from $\Sigma$. We also need to construct an analogous solution for the formal adjoint of equation \eqref{heat} that is given by 
\begin{equation}\label{heat_adj}
\begin{aligned}
\begin{cases}
-\p_t(\rho\, w)-\nabla\cdot(\sigma\nabla w)=0 
&\text{on  $M^{\textrm{int}}$},
\\
w=h\in \mathcal H^\star(\Sigma) & \text{on $\Sigma=(0,T)\times \p \Omega$},
\\
w|_{t=T}= 0 & \text{on $\Omega$},
\end{cases}
\end{aligned}
\end{equation}
where 
\begin{equation}\label{Hstar}
\mathcal H^\star(\Sigma)= \{h\in H^{\frac{3}{4},\frac{3}{2}}(\Sigma))\,:\, h|_{t=T}=0\}.
\end{equation}
Let us denote by $\Lambda^\star_{\rho,\sigma}$ the Dirichlet-to-Neumann map associated to \eqref{heat_adj} that is defined by
$$\Lambda^\star_{\rho,\sigma}(h)=(\sigma\p_\nu w)|_{\Sigma}.$$

We proceed to construct a special family of solutions to equation \eqref{heat} and its formal adjoint, namely \eqref{heat_adj}, by making the ansatz
\begin{equation}\label{u_form} 
u=v_{+}(t,x)+r_+(t,x),
\end{equation}
and
\begin{equation}\label{w_form} 
w=v_{-}(t,x)+r_-(t,x),\end{equation}
respectively. Here, $v_{\pm}$ are approximate solutions to \eqref{heat} and \eqref{heat_adj} respectively that are localized near the point $(t_0,p_0)$ and $r_{\pm}$ are remainder terms in $M$ and will be small compared to $v_{\pm}$. As $v_{\pm}$ is localized near the point $(t_0,p_0)$ we will construct it via the normal coordinate system $(t,y_1,y')$ above. 

\begin{remark}
	In what follows and for the sake of brevity, we will slightly abuse the notation and associate functions with their representations in the $(t,y)$-coordinate system, for example writing $v_{\pm}(t,y)$ instead of $v_{\pm}(t,Y^{-1}(y))$. 
\end{remark}

We let $\tau>0$ and make the ansatz
\begin{equation}\label{v^+_form} 
v_{+}(t,y)=e^{\textrm{i}\tau t+\tau \Psi(y)}(\underbrace{\sum_{k=0}^N \tau^{-k}\,v_{+,k}(t,y) }_{a_{+}(t,y)})\,\chi_{0}(\epsilon^{-1}\,y_1),
\end{equation}
and
\begin{equation}\label{v^-_form}
v_{-}(t,y)=e^{-\textrm{i}\tau t+\tau \overline{\Psi(y)}}(\underbrace{\sum_{k=0}^N \tau^{-k}\,v_{-,k}(t,y)}_{a_{-}(t,y)})\,\chi_{0}(\epsilon^{-1}y_1),
\end{equation}
where $\textrm{i}$ is the purely imaginary unit number, $N\in \N$ is a fixed positive integer, $\Psi$ is called the phase function, $v_{\pm,k}$, $k=0,1,\ldots,N$, are called amplitude functions that will be compactly supported in a neighborhood of the point $(t_0,0)$ on $\Sigma$, $\epsilon$ is as in \eqref{epsilon_def} and finally $\chi_0:\R\to\R$ is a smooth non-negative function that satisfies $\chi_0(t)=1$ for all $|t|\leq \frac{1}{4}$ and $\chi(t)=0$ for all $|t|\geq \frac{1}{2}$. 

Our construction of the phase and amplitude functions are based on performing a Wentzel--Kramers--Brillouin (WKB in short) construction with respect to the semi-classical parameter $\frac{1}{\tau}$ for the conjugated heat operators
$$ P_+v=e^{-\textrm{i}\tau t-\tau \Psi}(\rho\p_t -\nabla\cdot \sigma\nabla)(e^{\textrm{i}\tau t+\tau\Psi}v)$$
and
$$ P_-v=e^{\textrm{i}\tau t-\tau\overline{\Psi}}(-\rho\p_t -\nabla\cdot \sigma\nabla-(\p_t\rho))(e^{-\textrm{i}\tau t+\tau\overline\Psi}v)$$
respectively. Expanding the previous expressions in the $y$-coordinate system and recalling that $\Psi=\Psi(y)$ is assumed to be independent of $t$, we obtain
\begin{multline*}
P_+v= -\sigma\tau^2\langle d\Psi,d\Psi\rangle_g\,v-2\tau\sigma \langle d\Psi,dv\rangle_g-\tau\langle d\sigma,d\Psi\rangle_g v\\
-\tau\sigma (\Delta_g\Psi)v+\textrm{i}\tau\rho v
+\rho\p_t v -\langle d\sigma,dv\rangle_g -\sigma \Delta_g v, 
\end{multline*}
and
\begin{multline*}
P_-v= -\sigma\tau^2\langle d\overline\Psi,d\overline\Psi\rangle_g\,v-2\tau\sigma \langle d\overline\Psi,dv\rangle_g-\tau\langle d\sigma,d\overline\Psi\rangle_g v\\
-\tau\sigma (\Delta_g\overline\Psi)v+\textrm{i}\tau\rho v
-\rho\p_t v -\langle d\sigma,dv\rangle_g -\sigma \Delta_g v-(\p_t\rho)v, 
\end{multline*}
We will choose the phase functio $\Psi$ and the amplitudes $v_{\pm,k}$, $k=0,\ldots,N$, by requiring that the expressions $P_\pm v_{\pm}$ vanish up to ${(N-1)}^{\textrm{th}}$ order on $y_1=0$ in a neighborhood of the point $(t_0,0)$.

\subsubsection{Construction of the phase function}
The construction of the phase function $\Psi(y)$ is carried out by first making the ansatz
\begin{equation}\label{phase_ansatz}
\Psi(y)=\sum_{\ell=0}^N \psi_{\ell}(y')\,(y_1)^\ell,\end{equation}
where $\psi_\ell(y')$ is a homogeneous polynomial of degree $\ell$ in the variables $y'=(y_2,\ldots,y_n)$. We require that
\begin{equation}
\label{phase_eq}
(\frac{\p^j}{\p y_1^j}\langle d\Psi,d\Psi\rangle_g)\big|_{y_1=0}=0 \quad \text{for $j=0,1,\ldots,N-1$}.\end{equation}
Observe that in view of \eqref{metric_form} there holds
\begin{equation}\label{phase_exp}
\langle d\Psi,d\Psi\rangle_g = (\p_{y_1}\Psi)^2+\sum_{\alpha,\beta=2}^ng'^{\alpha\beta}(y_1,y')\p_{y_\alpha}\Psi\,\p_{y_\beta}\Psi.
\end{equation}
In particular, for $j=0$ the equation \eqref{phase_eq} reduces to
\begin{equation}\label{random}\psi_1^2 + \sum_{\alpha,\beta=2}^ng'^{\alpha\beta}(0,y')\p_{y_\alpha}\psi_0\,\p_{y_\beta}\psi_0=0.\end{equation}
We choose 
\begin{equation}
\label{phi_0}
 \psi_0(y')=\textrm{i}\,\psi(y')\quad \text{where}\quad \sum_{\alpha,\beta=2}^ng'^{\alpha\beta}(0,y')\p_{y_\alpha}\psi\,\p_{y_\beta}\psi=1
 \end{equation}
in a neighborhood of the origin. In view of \eqref{random} we may choose
\begin{equation}
\label{phi_1}
\psi_1(y')=1.
\end{equation}
It is straightforward to see from equations \eqref{phase_eq}--\eqref{phase_exp} that the subsequent terms $\psi_\ell$ with $\ell=2,\ldots,N$ are going to be iteratively determined via algebraic expressions in terms of the previous terms $\psi_j$, $j=0,1,\ldots,\ell-1$. Note also that the phase function is explicitly known in the sense that it does not depend on $\sigma$ and $\rho$.

\subsubsection{Construction of the principal amplitude}
\label{sec_pr}
The construction of the principal amplitude term $v_{\pm,0}$ is carried out by first making the ansatz
\begin{equation}\label{p_amp}
v_{\pm,0}(t,y)= \sum_{\ell=0}^{N}v_{\pm,0,\ell}(t,y')(y_1)^\ell,\end{equation}
where $v_{\pm,0,\ell}(t,y')$ is a homogeneous polynomial of degree $\ell$ in the variables $(t,y')=(t,y_2,\ldots,y_n)$. We will require that  
\begin{equation}
\label{v_eq_+}
(\frac{\p^k}{\p  y_1^k}\left(2\sigma \langle d\Psi,dv_{+,0}\rangle_g+(\langle d\sigma,d\Psi\rangle_g+\sigma (\Delta_g\Psi)-\textrm{i}\rho) v_{+,0}\right))\big|_{y_1=0}=0,\end{equation}
and
\begin{equation}\label{v_eq_-}
(\frac{\p^k}{\p y_1^k}\left(2\sigma \langle d\overline\Psi,dv_{-,0}\rangle_g+(\langle d\sigma,d\overline\Psi\rangle_g +\sigma (\Delta_g\overline\Psi)-\textrm{i}\rho )v_{-,0}\right))\big|_{y_1=0}=0,\end{equation}
for $k=0,1,\ldots,N-1$. For $\delta \in (0,\epsilon)$ sufficiently small (independent of $\tau$), we set 
\begin{equation}
\label{v_pm_0}
v_{\pm,0,0}(t,y')=\chi_\delta(t,y')=\delta^{-\frac{n}{2}} \chi_0(\delta^{-1}\sqrt{(t-t_0)^2+y_2^2+\ldots+y_n^2}),
\end{equation}
where as described earlier $\chi_0:\R\to\R$ is a smooth non-negative function that satisfies $\chi_0(t)=1$ for all $|t|\leq \frac{1}{4}$ and $\chi(t)=0$ for all $|t|\geq \frac{1}{2}$. Note that by applying equations \eqref{v_eq_+}--\eqref{v_eq_-} with $k=0$, we obtain
\begin{multline*}
2\sigma\,v_{+,0,1}+(\p_{y_1}\sigma)\,v_{+,0,0}+2\textrm{i}\sigma\langle d\psi,dv_{+,0,0}\rangle_{g'(0,y')}\\
+(\textrm{i}\langle d\sigma,d\psi \rangle_{g'(0,y')}+\sigma(\Delta_g \Psi)-\textrm{i}\rho)v_{+,0,0}=0, 
\end{multline*}
and
\begin{multline*}
2\sigma\,v_{-,0,1}+(\p_{y_1}\sigma)\,v_{-,0,0}-2\textrm{i}\sigma\langle d\psi,dv_{-,0,0}\rangle_{g'(0,y')}\\
(-\textrm{i}\langle d\sigma,d\psi \rangle_{g'(0,y')}+\sigma(\Delta_g \overline\Psi)-\textrm{i}\rho)v_{-,0,0}=0,
\end{multline*}
where it is to be understood that all the coefficients in the latter two equations are evaluated at points $(t,0,y')$. Thus, $v_{\pm,0,1}(t,y')$ is uniquely determined, it is supported in a $\delta$-neighborhood of $(t_0,0,0)$ and its value depends only on $\sigma|_{\Sigma}$, $\p_{y_1}\sigma|_{\Sigma}$ and $\rho|_{\Sigma}$. It follows in a similar manner that the remaining terms $v_{\pm,0,\ell}(t,y')$, $\ell=2,\ldots,N$, will be iteratively determined from equations \eqref{v_eq_+}--\eqref{v_eq_-} and that all of them are supported in a $\delta$-neighborhood of the point $(t_0,0,0)$ on $\Sigma$. Additionally, there holds
\begin{equation}
\label{v_0k_form}
v_{\pm,0,\ell}(t,y')=-\frac{1}{2\,\sigma\,\ell!}(\p^\ell_{y_1}\sigma - \textrm{i}\,\p_{y_1}^{\ell-1}\rho)\,v_{\pm,0,0}(t,y')+ \tilde{v}_{\pm,0,\ell}(t,y')
\end{equation}
where $\sigma^{-1}$ and the term in the parenthesis on the right hand side are evaluated at $y_1=0$ and also that there holds
\begin{itemize}
	\item [(A1)]{ The term $\tilde{v}_{\pm,0,\ell}(t,y')$, $\ell=1,\ldots,N$, only depends on $M$ and the quantities $\{\p^j_{y_1}\sigma(t,0,y')\}_{j=0}^{\ell-1}$ and $\{\p^j_{y_1}\rho(t,0,y')\}_{j=0}^{\ell-2}$.} 
	\end{itemize}

\subsubsection{Construction of the subprincipal amplitudes}
\label{sec_subp}
The construction of the subprincipal amplitude terms $v_{\pm,k}$, $k=1,\ldots,N$ is carried out by first making the ansatz
\begin{equation}\label{subp_amp}
v_{\pm,k}(t,y)= \sum_{\ell=0}^{N}v_{\pm,k,\ell}(t,y')(y_1)^\ell,
\end{equation}
where $v_{\pm,k,\ell}(t,y')$ is a homogeneous polynomial of degree $\ell$ in the variables $(t,y')=(t,y_2,\ldots,y_n)$. We will require that  
\begin{multline}\label{subp_eq_+}
-\frac{\p^j}{\p y_1^j}\left(2\sigma\langle d\Psi,dv_{+,{k}}\rangle_g+(\langle d\sigma,d\Psi\rangle_g+\sigma (\Delta_g\Psi)-\textrm{i}\rho)v_{+,k}\right)|_{y_1=0}\\
+\frac{\p^j}{\p y_1^j}\left(\rho\p_tv_{+,k-1}-\langle d\sigma,dv_{+,k-1} \rangle_g-\sigma \Delta v_{+,k-1}\right)|_{y_1=0}
=0,
\end{multline}
and
\begin{multline}\label{subp_eq_-}\frac{\p^j}{\p y_1^j}\left(2\sigma\langle d\overline\Psi,dv_{-,k}\rangle_g+(\langle d\sigma,d\overline\Psi\rangle_g+\sigma (\Delta_g\overline\Psi)-\textrm{i}\rho)v_{-,{k}}\right)|_{y_1=0}\\
+\frac{\p^j}{\p y_1^j}\left(\rho\p_tv_{-,k-1}+\langle d\sigma,dv_{-,k-1} \rangle_g+\sigma \Delta v_{-,k-1}+(\p_t\rho)v_{-,k-1}\right)|_{y_1=0}
=0,
\end{multline}
for all $j=0,1,\ldots,N-1$ and $k=1,\ldots,N$. As in the case of principal amplitudes, we set 
\begin{equation}
\label{v_pm_k_0}
v_{\pm,k,0}(t,y')=\chi_\delta(t,y'),
\end{equation}
and proceed to determine the remaining terms $v_{\pm,k,\ell}(t,y')$, $\ell=1,\ldots,N$. Analogously to the case of the principal amplitudes, it follows from equations \eqref{subp_eq_+}--\eqref{subp_eq_-} that the terms $v_{\pm,k,\ell}$, $\ell=1,\ldots,N$, will be iteratively determined uniquely via algebraic expressions. It is also straightforward to see that all of the terms $v_{\pm,k,\ell}(t,y')$ are supported in a $\delta$-neighborhood of the point $(t_0,0)$ on $\Sigma$.

In our subsequent boundary determination analysis a more detailed description of $v_{\pm,k,1}$ is needed in particular. Therefore we record that for any $k\geq 1$ there holds
\begin{equation}
\label{v_k_l_form}
v_{\pm,k,1}(t,y')= (-\frac{1}{2})^{k+1}\,\sigma^{-1}\,(\p_{y_1}^{k+1}\sigma -\textrm{i}\p_{y_1}^{k}\rho)\,\chi_\delta(t,y')+\tilde{v}_{\pm,k,1}(t,y')
\end{equation}
where $\sigma^{-1}$ and the term in the parenthesis on the right hand side are evaluated on $y_1=0$ and also that there holds 
\begin{itemize}
	\item [(A2)]{The term $\tilde{v}_{\pm,k,1}(t,y')$, $k=0,\ldots,N$, depends only on $M$ and the quantities $\{\p^j_{y_1}\sigma(t,0,y')\}_{j=0}^{k}$ and $\{\p^j_{y_1}\rho(t,0,y')\}_{j=0}^{k-1}$.} 
\end{itemize}

\subsubsection{Remainder estimates}

To conclude our construction of the special solutions \eqref{u_form}--\eqref{w_form} to equations \eqref{heat} and \eqref{heat_adj}, we consider the following initial boundary value problems
\begin{equation}\label{heat_remainder}
\begin{aligned}
\begin{cases}
\rho(t,x)\,\p_t r_+-\nabla\cdot(\sigma(t,x)\nabla r_+)=F_+ 
&\text{on  $M^{\textrm{int}}$},
\\
r_+= 0 & \text{on $\Sigma=(0,T)\times \p \Omega$},
\\
r_+|_{t=0}= 0 & \text{on $\Omega$}.
\end{cases}
\end{aligned}
\end{equation}
and
\begin{equation}\label{heat_adj_remainder}
\begin{aligned}
\begin{cases}
-\p_t(\rho(t,x)\,r_-)-\nabla\cdot(\sigma(t,x)\nabla r_-)=F_- 
&\text{on  $M^{\textrm{int}}$},
\\
r_-= 0 & \text{on $\Sigma=(0,T)\times \p \Omega$},
\\
r_-|_{t=T}= 0 & \text{on $\Omega$}.
\end{cases}
\end{aligned}
\end{equation}
where
$$F_+=-e^{\textrm{i}\tau t+\tau\Psi}\,P_+v_{+}\quad \text{and}\quad F_-=-e^{-\textrm{i}\tau t+\tau\overline\Psi}\,P_-v_{-}.$$

In view of the definition of the amplitude terms \eqref{v^+_form}--\eqref{v^-_form}, we note that $F_\pm$ are supported in a small neighborhood in $M$ (depending on $\epsilon$) of the point $(t_0,p_0)$. Moreover, by combining the definition of $P_\pm$ together with equations \eqref{phase_eq} for the phase function, \eqref{v_eq_+}--\eqref{v_eq_-} for principal amplitudes and \eqref{subp_eq_+}--\eqref{subp_eq_-} for subprincipal amplitudes, and finally that $\chi_{0}(\epsilon^{-1}y_1)$ is equal to one in a neighborhood of the boundary $\Sigma$, we deduce that
$$ |F_\pm| \leq C\, (\tau^{2}\,|y_1|^N+\tau^{-N})\, e^{\frac{1}{2}\tau y_1},$$
for some constant $C>0$ that is independent of $\tau>0$. As $y_1\leq 0$ inside $M$ and by applying standard energy estimates for parabolic equations, we arrive at the following bounds for the remainder terms $r_\pm$ on $M$,
\begin{equation}
\label{r_bounds}
\|r_\pm\|_{H^{1,2}(M)}+\|\p_\nu r_\pm|_{\Sigma}\|_{H^{\frac{1}{4},\frac{1}{2}}(\Sigma)}\leq C \tau^{\frac{3}{2}-N},
\end{equation}
for some constant $C>0$ that is independent of $\tau>0$. Here the Sobolev spaces on the left hand side are as defined by \eqref{Sobolev_def_1}--\eqref{Sobolev_def_2}.

\subsection{Determination of $\sigma|_{\Sigma}$ from $\Lambda_{\rho,\sigma}$}
\label{sec_sigma}

We have the following lemma.
\begin{lemma}
	\label{boundary_int_0}
	Given any $f \in \mathcal H(\Sigma)$ and $h \in \mathcal H^\star(\Sigma)$, there holds
	$$\int_{\Sigma}\Lambda_{\rho,\sigma}(f)\,h\,dV_{\Sigma} =\int_{\Sigma}f\,\Lambda^\star_{\rho,\sigma}(h)\,dV_{\Sigma},$$
	where $dV_{\Sigma}$ is the volume measure on $\Sigma$. 
\end{lemma}

\begin{proof}
	Let $u$ and $w$ denote the unique solution to equations \eqref{heat}--\eqref{heat_adj} with Dirichlet boundary data $f$ and $h$ respectively. Then,
	\begin{multline*}
	0= \int_{M}(\rho\,\p_t u-\nabla\cdot(\sigma\nabla u))\,w\,dt\,dx-\int_{M}(-\rho\,\p_t w-\nabla\cdot(\sigma\nabla w)-\p_t\rho\, w)\,u\,dt\,dx \\
	=\int_M \p_t(\rho\, u\,w)\,dt\,dx +\int_{M}(w\nabla\cdot(\sigma\nabla u)-u\nabla\cdot(\sigma\nabla w))\,dt\,dx\\
	=\int_{M}(w\nabla\cdot(\sigma\nabla u)-u\nabla\cdot(\sigma\nabla w))\,dt\,dx\\
	=\int_{\Sigma} \sigma\,(w\,\p_\nu u-u\,\p_\nu w)\,dV_\Sigma=\int_{\Sigma}\Lambda_{\rho,\sigma}(f)\,h\,dV_{\Sigma}-\int_{\Sigma}f\,\Lambda^\star_{\rho,\sigma}(h)\,dV_{\Sigma}.
	\end{multline*} 
	
\end{proof}

For $j=1,2,$ let us assume that $\rho_j,\sigma_j$ are smooth positive functions on $M$ and that $\Lambda_{\rho_1,\sigma_1}=\Lambda_{\rho_2,\sigma_2}$. We proceed to show that $\sigma_1=\sigma_2$ on $\Sigma$. To this end, let us consider a fixed point $(t_0,p_0)\in \Sigma$ as before and consider $u_1$ to be the oscillatory localized solution to equation \eqref{heat} with $\sigma=\sigma_1$ and $\rho=\rho_1$ that concentrate near the point $(t_0,p_0)\in \Sigma$. We also denote by $u_2$ and $w_2$ respectively the oscillatory localized solutions to equations \eqref{heat} and \eqref{heat_adj} with $\sigma=\sigma_2$ and $\rho=\rho_2$ that are localized near the point $(t_0,p_0)\in \Sigma$.

For notational convenience, we will denote the dependence of the special solutions to $\sigma_j$ and $\rho_j$ via a superscript $(j)$, thus writing
\begin{equation}\label{u_j} u_j= v_{+}^{(j)}+r_{+}^{(j)}=e^{\textrm{i}\tau t+\tau \Psi}(\underbrace{\sum_{k=0}^N \tau^{-k}\,v^{(j)}_{+,k}}_{a^{(j)}_+})\,\chi_{0}(\epsilon^{-1}y_1)+r_+^{(j)},\quad j=1,2,\end{equation}
while
\begin{equation}\label{w_2}
w_2= v_{-}^{(2)}+r_{-}^{(2)}=e^{-\textrm{i}\tau t+\tau \overline{\Psi}}(\underbrace{\sum_{k=0}^N \tau^{-k}\,v^{(2)}_{-,k}}_{a_-^{(2)}} )\,\chi_{0}(\epsilon^{-1}y_1)+r_-^{(2)},\end{equation}
where the construction of phase and amplitude terms in the latter expressions are carried out as in the previous sections. Finally, we define 
\begin{equation}\label{fh} 
f_j= u_j|_{\Sigma}, \quad j=1,2, \quad \text{and}\quad h=w_2|_{\Sigma}.\end{equation}
Observe that (by taking $\delta\in (0,\epsilon)$ sufficiently small independent of $\tau$) there holds $f_j \in \mathcal H(\Sigma)$ and $h\in \mathcal H^\star(\Sigma).$
In view of Lemma~\ref{boundary_int_0} and the equality $\Lambda_{\rho_1,\sigma_1}=\Lambda_{\rho_2,\sigma_2}$ we have the following identity
\begin{equation}
\label{boundary_iden_0}
\int_{\Sigma} \sigma_1 \p_\nu u_1\,w_2\,dV_{\Sigma}=\int_{\Sigma} \sigma_2 \p_\nu w_2\,u_1\,dV_{\Sigma}
\end{equation}
Recalling that $y_1=0$ on $\Sigma$ together with the remainder estimates \eqref{r_bounds}, we deduce that
\begin{equation}
\label{iden_1}
|\int_{\Sigma} \sigma_1 \p_\nu v_{+}^{(1)}\,v_{-}^{(2)}\,dV_{\Sigma}-\int_{\Sigma} \sigma_2 \p_\nu v_-^{(2)}\,v_+^{(1)}\,dV_{\Sigma}|\leq C\,\tau^{\frac{5}{2}-N},
\end{equation}
for some constant $C>0$ that is independent of $\tau>0$. As the integrands in the left hand side of the latter inequality are supported in a small neighborhood of the point $(t_0,p_0)$ we may work in the $(t,y)$-coordinate system and recall that $p_0$ is the origin in $y$-coordinates. Using \eqref{u_j}--\eqref{w_2} together with the definitions of the phase and amplitude functions, we deduce that
\begin{multline}
\label{iden_2}
\int_{\Sigma} (\sigma_1 \p_\nu v_{+}^{(1)}\,v_{-}^{(2)}\,dV_{\Sigma}-\sigma_2 \p_\nu v_-^{(2)}\,v_+^{(1)})\,dV_{\Sigma}\\
=\tau \int_0^T\int_{y_1=0} (\sigma_1(t,0,y')-\sigma_2(t,0,y')
)\chi_\delta^2(t,y')\,dt\,dV_{g'(0,y')} + O(1),
\end{multline}
where we recall that $g'(0,y')$ is the induced metric on $\p\Omega$ and $dV_{g'(0,y')}$ is its volume form. Here, $O(1)$ means uniformly bounded in $\tau>0$. Taking $N=2$, and taking the limit as $\tau \to \infty$ it follows from \eqref{iden_1}--\eqref{iden_2} that there holds
$$ \int_0^T\int_{y_1=0} (\sigma_1(t,0,y')-\sigma_2(t,0,y')
)\chi_\delta^2(t,y')\,dt\,dV_{g'(0,y')}=0.$$
Finally, taking the limit $\delta\to 0$ in the latter expression and recalling the definition of $\chi_\delta$, we derive
$$ \sigma_1(t_0,p_0)=\sigma_2(t_0,p_0).$$
As $(t_0,p_0)\in \Sigma$ was arbitrary, we conclude that 
$$\sigma_1|_{\Sigma}=\sigma_2|_{\Sigma}.$$

\subsection{Determination of the Taylor series of $\rho$ and $\sigma$ on $\Sigma$}
We are ready to complete the boundary determination of coefficients via an induction argument on the order of differentiation in the $y_1$ coordinate at the points $y_1=0$.

\begin{proof}[Proof of Proposition~\ref{prop_boundary}]
We will prove the claim via induction. Our induction hypothesis is that given any $k\in \{0\}\cup\N$, there holds 
\begin{equation}\label{ind}
\p_\nu^{l} \sigma_1|_{\Sigma}=\p_\nu^{l} \sigma_1|_{\Sigma}\quad \text{and}\quad \p_\nu^{l-1} \rho_1|_{\Sigma}=\p_\nu^{l-1} \rho_2|_{\Sigma},\end{equation}
for any $l=0,\ldots,k$, where we remark that in the case $k=0$, the equality involving $\rho_1,\rho_2$ in the latter equation is void. We have already proved the base of this induction in the last section by establishing $\sigma_1|_{\Sigma}=\sigma_2|_{\Sigma}$. We assume that the claim is true for $k=m$ and proceed to prove the claim for $k=m+1$. 

Analogously to Section~\ref{sec_sigma}, we consider a point $(t_0,p_0)\in \Sigma$ and denote by $y=(y_1,\ldots,y_n)$, the normal coordinate system defined in a neighborhood of the point $p_0$ as in \eqref{fermi} and with $p_0$ being the origin in $y$-coordinates. We define $u_1$, $u_2$ and $w_2$ as in \eqref{u_j}--\eqref{w_2} with 
$$N=m+3.$$ 
The functions $f_1$, $f_2$ and $h$ are also given as in \eqref{fh}. As $u_2$ and $w_2$ solve equations \eqref{heat} and \eqref{heat_adj} with $\sigma=\sigma_2$ and $\rho=\rho_2$, there holds
$$ \int_{\Sigma} \sigma_2 \p_\nu u_2\,w_2\,dV_{\Sigma}=\int_{\Sigma} \sigma_2 \p_\nu w_2\,u_2\,dV_{\Sigma}.$$
Combining this identity with \eqref{boundary_iden_0} and the equality $\sigma_1|_{\Sigma}=\sigma_2|_{\Sigma}$, we obtain the following identity
$$ \int_{\Sigma} \sigma_1 (\p_\nu u_1-\p_\nu u_2)\,w_2\,dV_{\Sigma}=\int_{\Sigma} \sigma_1\,(u_1-u_2)\,\p_\nu w_2\,dV_{\Sigma}.$$
Applying the remainder estimates \eqref{r_bounds} with $N=m+3$, the latter identity reduces to
$$
\left|\int_{\Sigma} \sigma_1\left((\p_\nu v_+^{(1)}-\p_\nu v_+^{(2)})\,v_-^{(2)}-(v_+^{(1)}-v_+^{(2)})\,\p_\nu v_-^{(2)}\right)\,dV_{\Sigma}\right|\leq C\tau^{-m-\frac{1}{2}},
$$
for some constant $C>0$ that is independent of $\tau>0$. Note that the left hand side of the latter bound is supported in a small neighborhood of the point $(t_0,p_0)$ (depending on $\delta$) and as such we may work directly in the $(t,y)$ coordinate system. Using the expressions \eqref{u_j}--\eqref{w_2} together with \eqref{phase_ansatz}--\eqref{phi_1}, noting that $y_1=0$ on $\Sigma$ and $\nu=\p_{y_1}$ and that $\chi_{0}(\epsilon^{-1}y_1)$ is equal to one in a neighborhood of $\Sigma$, we can simplify the latter bound as follows
\begin{multline}
\label{iden_3}
|\int_0^T\int_{y_1=0} \sigma_1\,(\p_{y_1}(a_+^{(1)}-a_+^{(2)}))\,a_-^{(2)}\,dV_{g'}\,dt\\
-\int_0^T\int_{y_1=0}\sigma_1\,(a_+^{(1)}-a_+^{(2)})\,\p_{y_1} a_-^{(2)}\,dV_{g'(0,y')}\,dt|\leq C\tau^{-m-\frac{1}{2}},
\end{multline} 
Finally, by recalling the definition of $a^{(j)}_{\pm}$ together with \eqref{v_pm_0} and \eqref{v_pm_k_0} we deduce that $a_+^{(1)}(t,0,y')=a_+^{(2)}(t,0,y')$ which further simplifies the latter bound to
\begin{equation}
\label{iden_4}
\left|\int_0^T\int_{y_1=0} \sigma_1\,(\p_{y_1}(a_+^{(1)}-a_+^{(2)}))\,a_-^{(2)}\,dV_{g'}\,dt\right|\leq C\tau^{-m-\frac{1}{2}},
\end{equation} 
Recall from the induction hypothesis that identity \eqref{ind} holds for all integers $l=0,\ldots,m$. Thus, in view of \eqref{v_0k_form} and \eqref{v_k_l_form} and properties (A1)--(A2) we have   
\begin{equation}\label{v_diff_eq} v_{\pm,l,1}^{(1)}(t,y')=v_{\pm,l,1}^{(2)}(t,y'),\end{equation}
for all $l=0,\ldots,m-1$ and also that $\tilde{v}_{\pm,m,1}^{(1)}=\tilde{v}_{\pm,m,1}^{(2)}$. Combining the latter equality with \eqref{v_k_l_form} (for $k=m$) we obtain
$$
v_{+,m,1}^{(1)}- v_{+,m,1}^{(2)}=(-\frac{1}{2})^{m+1}\sigma_1^{-1}(\p_{y_1}^{m+1}(\sigma_1-\sigma_2)- \textrm{i}\,\p_{y_1}^{m}(\rho_1-\rho_2))\,\chi_{\delta}.\
$$ 
Returning to the bound \eqref{iden_4} and applying the latter equality together with \eqref{v_diff_eq} and recalling that $N=m+3$, the expression \eqref{iden_4} reduces to 
 $$
     \left|\int_0^T\int_{y_1=0} (\p_{y_1}^{m+1}(\sigma_1-\sigma_2)- \textrm{i}\,\p_{y_1}^{m}(\rho_1-\rho_2))\,\chi_\delta^2\,dV_{g'(0,y')}\,dt\right|\leq C\tau^{-\frac{1}{2}},
    $$
     for all $\tau>0$ sufficiently large, which implies that
             $$\int_0^T\int_{y_1=0} (\p_{y_1}^{m+1}(\sigma_1-\sigma_2)- \textrm{i}\,\p_{y_1}^{m}(\rho_1-\rho_2))\,\chi_\delta^2\,dV_{g'(0,y')}\,dt=0.$$
             Taking the limit as $\delta\to 0$ and recalling the definition of $\chi_\delta$, we deduce that the expression
             $$\p_{y_1}^{m+1}(\sigma_1-\sigma_2)- \textrm{i}\,\p_{y_1}^{m}(\rho_1-\rho_2)$$
             must vanish at the point $(t_0,0,0)$. As both the real part and imaginary part must also vanish at the point $(t_0,0)$ we obtain the induction claim for $k=m+1$ at the arbitrary point $(t_0,p_0)\in \Sigma$.  
\end{proof}

Note that as a consequence of Proposition~\ref{prop_boundary}, we may consider the following smooth extension of the coefficients to $[0,T]\times\R^n$ that we will fix throughout the remainder of this paper.

\begin{definition}
	\label{def_coefficients}
	Let the hypotheses of Theorem~\ref{thm_main} be satisfied. In view of Proposition~\ref{prop_boundary} and for $j=1,2$, we define extended smooth real valued functions $\sigma_j$, $\rho_j \in C^{\infty}([0,T]\times \R^n)$ such that
	\begin{itemize}
		\item [(i)]{$\rho_j,\sigma_j>0$ on $[0,T]\times\R^n$.}
		\item [(ii)]{$\rho_1=\rho_2$ and $\sigma_1=\sigma_2$ on $[0,T]\times (\R^n\setminus \Omega)$.}
	\end{itemize}
\end{definition}

\section{Reformulation of the inverse problem}

Let $T>0$ and let $\Omega\subset \R^n$, $n\geq 2$, be a bounded domain with a smooth boundary. The aim of this section is to show that it is possible, via standard change of variables, to reformulate our inverse problem for equation \eqref{heat} to an inverse problem related to a simplified version of \eqref{heat} where $\sigma=1$, and where an additional zeroth order coefficient appears in the equation. To this end, let $\gamma>0$ and $q$ be smooth functions on $M=[0,T]\times\overline \Omega$ and consider the initial boundary value problem

\begin{equation}\label{heat_alt}
\begin{aligned}
\begin{cases}
\gamma(t,x)\,\p_t u-\Delta u+q(t,x) u=0 
&\text{on  $M^{\textrm{int}}$},
\\
u= f & \text{on $\Sigma$},
\\
u|_{t=0}= 0 & \text{on $\Omega$}.
\end{cases}
\end{aligned}
\end{equation}
We define the Dirichlet-to-Neumann map associated to \eqref{heat_alt} via the mapping
\begin{equation}\label{DN_alt}
\Gamma_{\gamma,q}(f)=\p_{\nu}u(t,x)|_{\Sigma},\qquad \forall\, f\in \mathcal H(\Sigma),
\end{equation}
where $u$ is the unique solution in the energy space \eqref{energy_space} to the initial boundary value problem \eqref{heat_alt}. We also consider the formal adjoint of equation \eqref{heat_alt} that is given by
\begin{equation}\label{heat_adjoint}
\begin{aligned}
\begin{cases}
-\gamma\,\p_t w-\Delta w+(q-\p_t \gamma)w=0 
&\text{on  $M^{\textrm{int}}$},
\\
w= h & \text{on $\Sigma$},
\\
w|_{t=T}= 0 & \text{on $\Omega$}.
\end{cases}
\end{aligned}
\end{equation}
The Dirichlet-to-Neumann map associated to \eqref{heat_adjoint} is defined via the mapping
\begin{equation}\label{DN_adj}
\Gamma^\star_{\gamma,q}(h)=\p_{\nu}w(t,x)|_{\Sigma},\qquad \forall\, h\in \mathcal H^\star(\Sigma),
\end{equation}
where $\mathcal H^\star(\Sigma)$ is as in \eqref{Hstar} and $w$ is the unique solution in the energy space \eqref{energy_space} to the initial boundary value problem \eqref{heat_adjoint} with Dirichlet boundary data $h\in \mathcal H^\star(\Sigma)$. We have the following lemma, whose proof we omit as it follows analogously to Lemma~\ref{boundary_int_0}.

\begin{lemma}
	\label{boundary_int}
	Let $\gamma>0$ and $q$ be smooth functions on $M$. Given any $f \in \mathcal H(\Sigma)$ and $h \in \mathcal H^\star(\Sigma)$, there holds
	$$\int_{\Sigma}\Gamma_{\gamma,q}(f)\,h\,dV_{\Sigma} =\int_{\Sigma}f\,\Gamma^\star_{\gamma,q}(h)\,dV_{\Sigma},$$
	where $dV_{\Sigma}$ is the volume measure on $\Sigma$. 
\end{lemma}

	

The link between equation \eqref{heat_alt} and our original equation \eqref{heat} is established via the following lemma. 

\begin{lemma}
	\label{lem_change_var}
	Let $\sigma,\rho$ be positive smooth functions on $M=[0,T]\times\overline\Omega$ and let us define
	\begin{equation}
	\label{change_variable}
	\gamma=\sigma^{-1}\,\rho\quad \text{and}\quad q=\frac{\Delta \sqrt{\sigma}}{\sqrt\sigma}-\frac{\rho}{2\sigma^2}\p_t\sigma.
	\end{equation}
There holds:
\begin{equation}
\label{DN_relation}
\Gamma_{\gamma,q}(f) = \frac{\p_\nu \sigma}{2\sigma}f +\sigma^{\frac{1}{2}}\Lambda_{\rho,\sigma}(\sigma^{-\frac{1}{2}}f).
\end{equation}	
\end{lemma}

\begin{proof}
	The lemma follows trivially from the observation that if $u$ solves
	$$\rho\p_t u -\nabla\cdot(\sigma\nabla u)=0\qquad \text{on $M^{\textrm{int}}$},$$
	then the scaled function $\tilde u =\sqrt\sigma\,u$
	solves the equation 
	$$\gamma\,\p_t \tilde u-\Delta \tilde u+q \tilde u=0\qquad 
	   \text{on  $M^{\textrm{int}}$}.$$
\end{proof}

We are ready to state a reformulation of our main result that will imply our main theorem as a consequence. 

\begin{theorem}
	\label{thm_alt}
	Let $T>0$ and let $\Omega\subset \R^n$, $n\geq 2$, be a simply connected bounded domain with a smooth boundary. For $j=1,2$, let $\gamma_j>0$ and $q_j$ be smooth real valued functions on $[0,T]\times\R^n$ such that
	\begin{equation}\label{ext_eq}
	\gamma_1=\gamma_2\quad \text{and}\quad q_1=q_2\quad \text{on $[0,T]\times (\R^n\setminus \Omega)$}.\end{equation}
	Let
	$$\g_j(t,x)=\gamma_j(t,x)\left((dx_1)^2+\ldots+(dx_n)^2\right),\quad j=1,2,\, t\in [0,T],\, x\in\R^n,$$
	and assume that for any $t\in [0,T]$, the Riemannian manifold $(\overline\Omega,\g_j(t,\cdot))$ is simple. If,
	\begin{equation}\label{Gamma_eq}\Gamma_{\gamma_1,q_1}(f)=\Gamma_{\gamma_2,q_2}(f)\qquad \forall\, f\in \mathcal H(\Sigma),\end{equation}
	then 
	$$\gamma_1(t,x)=\gamma_2(t,x)\quad \text{and}\quad q_1(t,x)=q_2(t,x)$$
	for all $(t,x)\in M$.
\end{theorem}

Theorem~\ref{thm_alt} will be proved in Sections ~\ref{sec_thm}--\ref{sec_thm_q}. Let us now show how Theorem~\ref{thm_alt} can be used to prove our main result, namely Theorem~\ref{thm_main}. 

\begin{proof}[Proof of Theorem~\ref{thm_main} via Theorem~\ref{thm_alt}]
We assume that the hypotheses of Theorem~\ref{thm_main} is satisfied. Note that in view of Proposition~\ref{prop_boundary} we may assume without any loss of generality that the coefficients $\sigma_j$, $\rho_j$ are extended to $[0,T]\times \R^n$ and that the extensions satisfy the properties described in Definition~\ref{def_coefficients}. For $j=1,2$, let us define  
\begin{equation}
\label{change_variable_1}
\gamma_j=\sigma_j^{-1}\,\rho_j\quad \text{and}\quad q_j=\frac{\Delta \sqrt\sigma_j}{\sqrt\sigma_j}-\frac{\rho_j}{2\sigma_j^2}\p_t\sigma_j.
\end{equation}
Applying Lemma~\ref{lem_change_var} we observe that \eqref{Gamma_eq} is satisfied. Hence, in view of Theorem~\ref{thm_alt} there holds:
	$$\gamma(t,x):=\gamma_1(t,x)=\gamma_2(t,x)\quad \text{and}\quad q(t,x):=q_1(t,x)=q_2(t,x)$$
for all $(t,x)\in [0,T]\times \R^n$. In view of the definition \eqref{change_variable_1}, the second equation in the latter equality can be rewritten as
\begin{equation}
\label{sigma_diff_iden}
\frac{\Delta \sqrt{\sigma_1}}{\sqrt{\sigma_1}}-\gamma\,\frac{\p_t\sqrt\sigma_1}{\sqrt{\sigma_1}}=\frac{\Delta \sqrt{\sigma_2}}{\sqrt{\sigma_2}}-\gamma\,\frac{\p_t\sqrt\sigma_2}{\sqrt{\sigma_2}},
\end{equation}
on $[0,T]\times \R^n$. For $j=1,2,$ let us define 
$$F_j(t,x)= \log(\sqrt{\sigma_j}(t,x))\quad \text{on $[0,T]\times \R^n$}.$$ 
Then \eqref{sigma_diff_iden} can be rewritten as
$$ 
\gamma\, \p_t F_1 -\Delta F_1 -\nabla F_1\cdot\nabla F_1=\gamma\, \p_t F_2 -\Delta F_2 -\nabla F_2\cdot\nabla F_2.
$$
Finally, defining $G=F_1-F_2$ on the set $[0,T]\times \R^n$, the latter equation may be reduced as 
$$
\gamma \p_t G - \Delta G - (\nabla F_1+\nabla F_2)\cdot \nabla G =0,\quad \text{on $[0,T]\times \R^n$}.
$$
Moreover, $G(t,x)=0$ for all $(t,x)\in [0,T]\times \R^n\setminus \Omega$. Using the unique continuation principle for parabolic equations, see e.g. \cite[Sections 1 and 4]{Lin}, it follows that $G=0$ on $M$. Therefore $\sigma_1=\sigma_2$ on $M$ and consequently as $\gamma=\frac{\rho_j}{\sigma_j}$, $j=1,2,$ it follows that $\rho_1=\rho_2$ on $M$.
\end{proof}

In the remainder of this paper we will prove Theorem~\ref{thm_alt}. 

\section{Energy estimates for conjugated heat operators}
\label{sec_carleman}
Throughout this section, we let $T>0$, and let $\Omega \subset \R^n$, $n\geq 2$, be a bounded domain with a smooth boundary. We assume that $\gamma\in C^{\infty}(M)$ is a positive function and finally that $q\in C^{\infty}(M)$ is real valued. Let us define
\begin{equation}
\label{L_+}
L_+u = \gamma\p_t u - \Delta u +q u,\end{equation}
and its formal adjoint given by
\begin{equation}
\label{L_-}
L_- u = -\gamma\p_tu - \Delta u +(q-\p_t\gamma) u.\end{equation}
In the remainder of this section, we prove the following proposition regarding heat operators $L_\pm$ that are conjugated with certain exponential weights. This proposition can be viewed as a limiting Carleman estimate for the conjugate heat operator (cf. \cite{DKSU}). Let us also mention that when $\gamma\equiv 1$, similar Carleman estimates have appeared in the literature with a linear time-independent phase function $\phi=\alpha\cdot x$ where $\alpha$ is a unit vector, see e.g. \cite{CaK,ChK}. Note, however, that as $\gamma$ is a general function in this paper, the phase function $\phi$ is both time-dependent also nonlinear.

\begin{proposition}
	\label{energy_prop}
	Let $\phi \in C^{\infty}(M)$ be real valued and satisfy the eikonal equation
	\begin{equation}
	\label{eikonal}
	|\nabla \phi|^2=\nabla\phi\cdot\nabla\phi = \gamma \qquad \text{on $M=[0,T]\times\overline\Omega$}.
	\end{equation} 
	Let $\tau>0$, $F\in L^{2}(M)$ be real valued and denote by $r_\pm\in H^{1,2}(M)$ the unique solution to the equations
	\begin{align}
	\label{heat_source}
	\begin{cases}
	e^{-\tau^2 t-\tau \phi}L_+(e^{\tau^2 t+\tau \phi}r_+)=F 
	&\text{on  $(0,T)\times\Omega$},
	\\
	r_+=0 & \text{on $\Sigma=(0,T)\times \p \Omega$},
	\\
	r_+|_{t=0}= 0 & \text{on $\Omega$},
	\end{cases}
	\end{align}
	and
	\begin{equation}\label{heat_source_adj}
	\begin{aligned}
	\begin{cases}
	e^{\tau^2 t+\tau \phi}L_-(e^{-\tau^2 t-\tau \phi}r_-)=F 
	&\text{on  $(0,T)\times\Omega$},
	\\
	r_-=0 & \text{on $\Sigma=(0,T)\times \p \Omega$},
	\\
	r_-|_{t=T}= 0 & \text{on $\Omega$}.
	\end{cases}
	\end{aligned}
	\end{equation} 
	Then, there exists $\tau_0>0$ and $C>0$ such that given any $\tau>\tau_0$, there holds:
	\begin{equation}
	\label{r_est}
	\tau\,\|r_\pm\|_{L^2(M)}+\tau^{-\frac{1}{2}}\,\|\p_\nu r_\pm\|_{H^{\frac{1}{4},\frac{1}{2}}(\Sigma)}\leq C\,\|F\|_{L^2(M)}.
	\end{equation}
\end{proposition}

\begin{remark}
	\label{real_remark}
	Let us remark that the assumption that $q$ and $F$ in the previous proposition are real valued is not important and an analogous statement can be proved for complex valued $q$ and $F$.  We merely impose this as as our initial model \eqref{heat} has real valued coefficients (and thus $q$ will also be real valued, see \eqref{change_variable}) and also as we never need to work with any complex valued solutions to \eqref{heat_alt}. This also simplifies the presentation of the proof of the proposition. 
\end{remark}

\begin{proof}[Proof of Proposition~\ref{energy_prop}]
For the sake of brevity, we prove the claim for the operator $L_+$. The analogous claim for $L_-$ follows by a similar analysis. We also write $r=r_+$ and $L=L_+$. Note that in view of \eqref{eikonal}, equation \eqref{heat_source} reduces to
	\begin{equation}\label{heat_source_conj}
\begin{aligned}
\begin{cases}
\gamma\p_t r - \Delta r +q r -\tau B r +\tau \,\gamma\, \p_t \phi\, r=F 
&\text{on  $(0,T)\times\Omega$},
\\
r=0 & \text{on $\Sigma=(0,T)\times \p \Omega$},
\\
r|_{t=0}= 0 & \text{on $\Omega$},
\end{cases}
\end{aligned}
\end{equation} 
where $B$ is the first order anti-symmetric operator (w.r.t the real inner product on $L^2(M)$) defined by
$$Br = 2\nabla\phi \cdot \nabla r + (\Delta\phi)r.$$
By the theory of forward problem for parabolic equations and using the notation \eqref{Sobolev_def_1}, we record that $r \in H^{1,2}(M)$, see e.g. \cite[Section 2]{LM}. 

Next, let $\lambda>0$ and multiply equation \eqref{heat_source_conj} with $e^{\lambda \phi}r$ and perform integration over $M$. As we will see, our claimed estimate \eqref{r_est} follows from a series of integration by parts for $\lambda$ sufficiently large and $\tau$ sufficiently larger than $\lambda$. We write
$$
\int_{M} \left(\gamma\p_t r - \Delta r +q r -\tau B r +\tau \gamma \,\p_t \phi\, r\right)e^{\lambda\phi}r\, dt\,dx = \int_{M} F\,e^{\lambda\phi}r\,dt\,dx,
$$
and subsequently break the integration on the left hand side into the following three components
\begin{align*}
\textrm{I}&= \int_M  (\gamma\p_t r+qr+\tau \gamma \,\p_t \phi\, r)\,e^{\lambda\phi}r\, dt\,dx,\\
\textrm{II}&=-\int_M (\Delta r)\,e^{\lambda\phi}r\, dt\,dx,\\
\textrm{III}&=-\tau\int_M (Br) \,e^{\lambda\phi}r\, dt\,dx.
\end{align*}
For $\textrm{I}$, we use the fact that $r|_{t=0}=0$ and write
\begin{multline*}\textrm{I} = \frac{1}{2}\int_{\Omega}\gamma(T,x)e^{\lambda \phi(T,x)}|r(T,x)|^2\,dx -\frac{1}{2}\int_M \p_t(e^{\lambda \phi}\gamma) |r|^2\,dt\,dx \\
+\tau \int_M \gamma\,\p_t\phi\,e^{\lambda\phi}|r|^2\,dt\,dx+\int_M q\,e^{\lambda\phi}|r|^2\,dt\,dx\\
\geq -\frac{1}{2}\int_M \p_t(e^{\lambda \phi}\gamma) |r|^2\,dt\,dx
+\tau \int_M \gamma\,\p_t\phi\,e^{\lambda\phi}|r|^2\,dt\,dx+\int_M q\,e^{\lambda\phi}|r|^2\,dt\,dx\\
\geq -(C_0+C_1\,\tau\,+C_2\,\lambda) \int_M e^{\lambda\phi}|r|^2\,dt\,dx,
\end{multline*}
for some positive constants $C_0$, $C_1$, $C_2$ that are independent of the parameters $\tau>0$ and $\lambda>0$. For $\textrm{II}$, noting that $r|_{\Sigma}=0$, we write
\begin{multline*}
\textrm{II}=-\int_M (\Delta r)\,e^{\lambda\phi}r\, dt\,dx=\int_M |\nabla r|^2 e^{\lambda\phi}\,dt\,dx+\frac{1}{2}\int_M \nabla r^2\cdot\nabla(e^{\lambda\phi})\,dt\,dx\\
= \int_M |\nabla r|^2 e^{\lambda\phi}\,dt\,dx-\frac{1}{2}\int_M |r|^2(\lambda^2 \gamma+\lambda \Delta\phi)e^{\lambda \phi}\,dt\,dx\\
\geq \int_M |\nabla r|^2 e^{\lambda\phi}\,dt\,dx -\lambda(C_3+C_4\lambda)\int_M e^{\lambda\phi}|r|^2\,dt\,dx,
\end{multline*}
for some positive constants $C_3$, $C_4$ that are independent of $\tau>0$ and $\lambda>0$. Finally, for $\textrm{III}$, we use the fact that $r|_{\Sigma}=0$ together with the fact that $B$ is anti-symmetric to deduce that
\begin{multline*}
\textrm{III}=-\tau\int_M (Br) \,e^{\lambda\phi}r\, dt\,dx=\tau\int_M \nabla\phi\cdot\nabla(e^{\lambda\phi})|r|^2\,dt\,dx\\
= \lambda\,\tau\, \int_M \gamma e^{\lambda\phi} |r|^2\,dt\,dx\geq C_5\,\lambda\,\tau\, \int_M e^{\lambda\phi} |r|^2\,dt\,dx,
\end{multline*}
for some positive constant $C_5>0$ independent of $\tau,\lambda>0$. Combining the latter estimates for \textrm{I}--\textrm{III} we deduce that
$$
\int_M F\, e^{\lambda\phi}r\,dt\,dx\geq \int_M |\nabla r|^2 e^{\lambda\phi}\,dt\,dx+A\,\int_M e^{\lambda\phi} |r|^2\,dt\,dx,  
$$
where
$$
A=C_5\,\tau\,\lambda-C_0-C_1\tau-(C_2+C_3)\lambda-C_4\lambda^2.$$
Thus there exists positive constants $\lambda_0$ and $\alpha$ only depending on $C_j$ with $j=0,\ldots,5$, such that given any $\lambda\geq \lambda_0$ and any $\tau\geq \alpha \lambda$, there holds  
$$
\int_M |\nabla r|^2 e^{\lambda\phi}\,dt\,dx+\frac{C_5}{2}\,\tau\,\lambda \int_M e^{\lambda\phi} |r|^2\,dt\,dx\leq \int_M F\, e^{\lambda\phi}\,r\,dt\,dx.  
$$
By applying Young's inequality on the right hand side of the latter expression with a fixed choice $\lambda\geq \max\{\lambda_0,\frac{4}{C_5}\}$, we deduce that there exists $\tau_0>0$ and $C>0$ both depending only on $M$, $q$ and $\gamma$ such that
\begin{equation}
\label{r_tau_bound}
\tau^{\frac{1}{2}}\|r\|_{L^2(0,T;H^1(\Omega))}+ \tau\,\|r\|_{L^2(M)}\leq C\|F\|_{L^2(M)},  
\end{equation}
for all $\tau\geq \tau_0$. Next, we rewrite equation \eqref{heat_source_conj} as
	\begin{equation}\label{heat_source_conj_alt}
\begin{aligned}
\begin{cases}
\gamma\p_t r - \Delta r +q r=G 
&\text{on  $M^{\textrm{int}}$},
\\
r=0 & \text{on $\Sigma=(0,T)\times \p \Omega$},
\\
r|_{t=0}= 0 & \text{on $\Omega$},
\end{cases}
\end{aligned}
\end{equation} 
where 
$$G= F +\tau\, Br -\tau\,\gamma\,\p_t\phi\,r.$$
In view of \eqref{r_tau_bound} there holds
$$ \|G\|_{L^2(M)}\leq C \tau^{\frac{1}{2}}\|F\|_{L^2(M)},$$
for all $\tau>\tau_0$ where $C>0$ is a constant depending only on $M$, $q$ and $\gamma$. Next, and by combining the latter Sobolev bound with standard Sobolev energy estimates for equation \eqref{heat_source_conj_alt} and the notations \eqref{Sobolev_def_1}--\eqref{Sobolev_def_2}, we obtain
\begin{equation}
\label{r_est_random_1}
\|r\|_{H^{1,2}(M)}+\|\p_\nu r\|_{H^{\frac{1}{4},\frac{1}{2}}(\Sigma)}\leq C\tau^{\frac{1}{2}}\|F\|_{L^2(M)},
\end{equation}
for some constant $C>0$ only depending on $M$, $q$ and $\gamma$. Finally, our claimed estimate \eqref{r_est} for $r_+$ follows from combining \eqref{r_tau_bound} and \eqref{r_est_random_1}.

We remark that in order to derive the desired estimate for $r_-$, one should multiply equation \eqref{heat_source_adj} with the alternative choice $e^{-\lambda \phi}\,r_-$ and proceed analogously as above. The claim \eqref{r_est} for $r_-$ then follows for $\lambda$ sufficiently large and $\tau$ sufficiently larger than $\lambda$.
\end{proof}

\section{Exponentially growing and decaying solutions}
\label{sec_exp}

We will assume throughout this section that $T>0$ and that $\Omega\subset \R^n$, $n\geq 2$, is a bounded simply connected domain with a smooth boundary. We also assume that $q$ is a smooth real valued function on $[0,T]\times \R^n$, that $\gamma>0$ is a smooth function on $[0,T]\times\R^n$ and finally that, writing 
$$\g(t,x)= \gamma(t,x)\left((dx_1)^2+\ldots+(dx_n)^2\right),$$
the following assumption is satisfied on $\Omega$:
\begin{itemize}
	\item [(H1)]{Given each $t \in [0,T]$, the Riemannian manifold $(\overline\Omega,\g(t,\cdot))$ is simple.}
\end{itemize}
The aim of this section is to construct a family of exponentially growing solutions to equation \eqref{heat_alt} as well as a family of exponentially decaying solutions to its formal dual, namely \eqref{heat_adjoint}. Both families of solutions will be {\em concentrated} along a fixed spacetime curve $\{(t_0,\zeta(s))\in M\,:\, s\in I\}$, where $t_0 \in (0,T)$ and $\zeta$ is any inextendible geodesic ray with respect to the Riemannian metric $\g(t_0,\cdot)$ on $\overline\Omega$ that connects two distinct boundary points on $\p \Omega$. 

To start the construction of the exponential solutions, we write $\tau>0$ to stand for a large asymptotic parameter and write $u_{\pm,\tau}$ to respectively stand for the exponentially growing and decaying solutions to \eqref{heat_alt} and \eqref{heat_adjoint} and make the following ansatz
\begin{equation}
\label{u_exp_pm}
u_{\pm,\tau}(t,x)= e^{\pm\tau^2 t\pm\tau \phi(t,x)}\left(\mu_{\pm}(t,x)+R_{\pm,\tau}(t,x)\right).
\end{equation}
Recalling the definitions \eqref{L_+}--\eqref{L_-}, we write for any $v\in C^{\infty}(M)$,
\begin{multline}
\label{L_conj_+}
e^{-\tau^2 t-\tau\phi}L_+(e^{\tau^2t+\tau \phi}v)=-\tau\,\left(2\nabla\phi\cdot \nabla v+(\Delta\phi-\gamma\p_t\phi)\,v\right)\\
+\tau^2\,(-\nabla\phi\cdot\nabla\phi+\gamma)\,v+(\gamma\p_t-\Delta+q)\,v,
\end{multline} 
and
\begin{multline}
\label{L_conj_-}
e^{\tau^2 t+\tau\phi}L_-(e^{-\tau^2t-\tau \phi}v)=\tau\,\left(2\nabla\phi\cdot \nabla v+(\Delta\phi+\gamma\p_t\phi)\,v\right)\\
+\tau^2\,(-\nabla\phi\cdot\nabla\phi+\gamma)\,v+(-\gamma\p_t-\Delta+q-\p_t\gamma)\,v,
\end{multline}
Motivated by the above expansions, we proceed with a WKB analysis in $\tau$ and require that $\phi \in C^{\infty}(M)$ and $\mu_{\pm}\in C^{\infty}(M)$,
are smooth real valued functions that are independent of the parameter $\tau$ and such that they satisfy the following differential equations on $M$ respectively,

\begin{align}
\label{phi_eikonal}
|\nabla\phi|^2 =\nabla\phi\cdot\nabla\phi=\gamma,
\end{align}
\begin{equation}
\label{mu_+}
\begin{aligned}
2\nabla \phi\cdot \nabla\mu_{+}+ (\Delta\phi-\gamma\,\p_t\phi)\,\mu_{+}&=0, 
\end{aligned}
\end{equation}
and
\begin{equation}
\label{mu_-}
\begin{aligned}
2\nabla \phi\cdot \nabla\mu_{-}+ (\Delta\phi+\gamma\,\p_t\phi)\,\mu_{-}&=0.
\end{aligned}
\end{equation}
We remark that for each fixed $t\in [0,T]$ the equations above can be viewed as eikonal and transport equations on $\Omega$. Before describing a canonical construction of solutions to \eqref{phi_eikonal}--\eqref{mu_-}, let us describe how the remainder terms $R_{\pm,\tau}$ can now be constructed with suitable decay estimates in $\frac{1}{\tau}$. To this end, we let $R_{\pm,\tau}$ solve the following initial boundary value problems
	\begin{equation}\label{heat_source_L}
\begin{aligned}
\begin{cases}
e^{\mp\tau^2 t\mp\tau \phi}L_\pm(e^{\pm\tau^2 t\pm\tau \phi}\,R_{\pm,\tau})=F_{\pm,\tau} 
&\text{on  $M^{\textrm{int}}$},
\\
R_{\pm,\tau}=0 & \text{on $\Sigma=(0,T)\times \p \Omega$},
\\
R_{\pm,\tau}|_{t=0}= 0 & \text{on $\Omega$},
\end{cases}
\end{aligned}
\end{equation} 
where
$$ F_{\pm,\tau} = - e^{\mp\tau^2 t\mp\tau \phi}L_\pm(e^{\pm\tau^2 t\pm\tau \phi}\,\mu_{\pm}).$$ 
In view of \eqref{phi_eikonal}--\eqref{mu_-}, it is straightforward to see that 
$$ F_{+,\tau}=-\left(\gamma\p_t-\Delta+q\right)\,\mu_{+},$$
and
$$
F_{-,\tau}=\left(\gamma\p_t+\Delta+\p_t\gamma-q\right)\,\mu_{-}.$$
Note that both functions are independent of the parameter $\tau>0$. Therefore, by combining the latter two equations with Proposition~\ref{energy_prop}, it follows that the remainder terms $R_{\pm,\tau}\in H^{1,2}(M)$ satisfy the bounds
\begin{equation}
\label{remainder_estimates}
\tau \,\|R_{\pm,\tau}\|_{L^2(M)}+\tau^{-\frac{1}{2}}\,\|\p_\nu R_{\pm,\tau}\|_{H^{\frac{1}{4},\frac{1}{2}}(\Sigma)} \leq C,
\end{equation}
for all $\tau>\tau_0$ and some constant $C>0$ that is independent of $\tau$. Thus, to complete the construction of the special solutions, it suffices to show how the eikonal and transport equations \eqref{phi_eikonal}--\eqref{mu_-} can be smoothly solved on $M$. This is in fact where the assumption (H1) will be used.  

\subsection{Canonical solution to eikonal and transport equations}
Let $t_0 \in (0,T)$, $x_0\in \R^n\setminus \overline\Omega$ be sufficiently close to $\p\Omega$ and finally let $\zeta_{t_0,x_0}$ be a unit speed geodesic segment (w.r.t the Riemannian metric $\g(t_0,\cdot)$) emanating from the point $x_0$ that passes transversally through $\Omega$ and terminates at a final point $y_0\in \p\Omega$. In what follows, we describe a canonical procedure to produce global smooth solutions to \eqref{phi_eikonal}--\eqref{mu_-} on $M$ that implicitly depend on the fixed parameters $t_0$, $x_0$, the segment $\zeta_{t_0,x_0}$ as well as a small fixed parameter $\delta>0$ that is independent of the parameter $\tau$. 

\subsubsection{Eikonal equation via a distance function with respect to the metric $\g$}
\label{sec_phi}
Note that since $x_0\in \R^n\setminus \overline\Omega$ is assumed to be sufficiently close to the boundary $\p\Omega$, it follows from (H1) that given any $x\in \overline\Omega$ (as well any $x$ in a small neighborhood of $\Omega$) and any $t\in [0,T]$ there exists a unique length minimizing unit speed geodesic $\eta_{t}:[0,L_{t,x}]\to \R^n$ with respect to the Riemannian metric $\g(t,\cdot)$ on $\R^n$ that satisfies $\eta_t(0)=x_0$ and $\eta_t(L_{t,x})=x$. Given any $t\in [0,T]$ and any $x\in \overline\Omega$ as well as any $x$ in a small neighborhood of $\Omega$, we now define 
\begin{equation}
\label{phase_sol}
\phi(t,x)= L_{t,x}=\text{Riemannian length of $\eta_t$ w.r.t $\g(t,\cdot)$}.
\end{equation} 
It is clear that $\phi$ is smooth and that it satisfies the eikonal equation \eqref{phi_eikonal} on $M$. Note that the construction of the phase function implicitly depends on the point $x_0$.

\subsubsection{Transport equation via normal polar coordinates with respect to the metric $\g$}
\label{sec_mu}
We will solve the transport equations \eqref{mu_+}--\eqref{mu_-} in a canonical way with solutions that are compactly supported in a tubular spacetime neighborhood of the geodesic segment $\zeta_{t_0,x_0}$ that is denoted by a small fixed parameter $\delta>0$ independent of $\tau$. In order to simplify the presentation of solutions to the transport equations, we will work in a special global coordinate system on $M$, that we describe next. Let us denote by $\mathbb S^{n-1}$ the unit sphere in $\R^n$ and define the mapping
$$ Z:[0,T]\times (0,\infty)\times \mathbb S^{n-1} \to [0,T]\times \R^n,$$
that is defined via
$$ Z(t,z_1,\underbrace{z_2,\ldots,z_n}_{z'})= (t,\textrm{exp}_{t,x_0}(z_1\, z')),$$
where $\textrm{exp}_{t,x_0}$ denotes the exponential map centered at $x_0$ with respect to the Riemannian metric $\g(t,\cdot)$ on $\R^n$ with $t\in [0,T]$. We write $U=Z^{-1}(M)$ and observe that in view of the assumption (H1), the restriction 
$$Z:U \to M,$$
is a diffeomorphism. This gives us the desired coordinate system on $M$ (which is also well defined in a small neighborhood of $M$ in $[0,T]\times \R^n$). By applying Gauss's lemma, we may also write
\begin{equation}
\label{norm_coord}
\g(t,z)= (dz_1)^2 + \sum_{\alpha=2}^{n}g'_{\alpha\beta}(t,z)\,dz_{\alpha}\,dz_{\beta},
\end{equation}
where we are abusing the notation slightly by writing $\g(t,z)$ in place of $Z^\star \g(t,\cdot)$. Observe that rewriting the phase function $\phi$ in the $(t,z)$-coordinates, there holds:
\begin{equation}
\label{phi_z}
\phi(t,z_1,z')= z_1.
\end{equation}
We refer the reader to \cite{Lee} for more details on normal polar coordinates.  Next, we note that the geodesic segment $\zeta_{t_0,x_0}$ (w.r.t $\g(t_0,\cdot)$) is represented in the $(t,z)$-coordinate system via 
\begin{equation}
\label{zeta_coord}
t=t_0,\qquad  a \leq z_1\leq b, \qquad z'=z'_0=(z'_{0,2},\ldots,z'_{0,n}), 
\end{equation}
for some positive constants $a,b>0$ and some constant $z'_0 \in \mathbb S^{n-1}$. Observe also that in $(t,z)$-coordinates there holds 
\begin{equation}\label{transport_z}
Z^\star(\gamma^{-1}\nabla \phi\cdot\nabla v)=\frac{\p}{\p z_1}v\circ Z.
\end{equation} 
In order to solve equations \eqref{mu_+}--\eqref{mu_-} on $M$ we first consider the same two equations in the $(t,z)$-coordinates. In order to solve the transport equations in the $(t,z)$-coordinates, we first consider them on the set $V$ defined by
\begin{equation}\label{V_def} 
V=\{t\in (t_0-\delta,t_0+\delta), \, a-\delta \leq z_1 \leq b+\delta,\, |z'-z'_0|< \delta\}.\end{equation}
We will construct solutions to \eqref{mu_+}--\eqref{mu_-} on the set $V$ that are in fact compactly supported in the wedge 
$$W=\{t\in (t_0-\delta,t_0+\delta), \, |z'-z'_0|< \delta\}.$$ 
It is then clear that the same solutions written in the original $(t,x)$-coordinates satisfy \eqref{mu_+}--\eqref{mu_-} on the set $M$, provided that $\delta>0$ is small enough depending on $t_0$ and $x_0$ and $\zeta_{t_0,x_0}$.

\begin{remark}
	\label{remark_notation} 
	In the remainder of this section and for the sake of brevity of notation, we will sometimes not explicitly write the pull back with respect to $Z$ explicitly. For example, we write $\mu_{\pm}$ to stand for the solutions in the $(t,z)$-coordinates, (i.e. $Z^\star \mu_{\pm}$). A similar understanding must be taken into account for other functions such as $\gamma$ in the remainder of this section. 
\end{remark}

In light of the above discussion, we begin by considering \eqref{mu_+}--\eqref{mu_-} in $(t,z)$-coordinates restricted to the set $V$. Dividing the two equations by $\gamma$ and in view of \eqref{transport_z}, both equations can be written in the form
\begin{equation}\label{mu_0_z} 
\p_{z_1} \mu_{\pm} + K_{\pm}\,\mu_{\pm} = 0 \quad \text{on $V$},
\end{equation}
where
$$K_{\pm}= \frac{1}{2}Z^\star\left(\gamma^{-1}\,(\Delta\phi\mp\gamma\p_t\phi))\right).$$
We note that $K_{\pm}$ is independent of $q$. To solve the latter pair of equations in \eqref{mu_0_z}, we first consider the unique solution $\tilde{\mu}_{\pm}$ to the transport equation \eqref{mu_0_z} on $V$ subject to the initial data
\begin{equation}
\label{mu_initial}
\tilde{\mu}_{\pm}(t,a-\delta,z') = 1 \quad (t,z')\in W,
\end{equation}
and subsequently define the compactly supported  (in $W$) solution 
\begin{equation}\label{mu_0_def} 
\mu_{\pm}(t,z_1,z') = \tilde{\mu}_{\pm}(t,z_1,z')\, \chi_\delta(t,z'), \quad \text{on $V$,}
\end{equation}
where the function $\chi_\delta:W\to [0,\infty)$ is defined via
\begin{equation}\label{chi_def} 
\chi_\delta(t,z')= \delta^{-\frac{n}{2}}\chi_0(\delta^{-1}\sqrt{ (t-t_0)^2+(z'_2-z'_{0,2})^2+\ldots+(z'_n-z'_{0,n})^2}),\end{equation}
where $\chi_0:\R\to \R$ is a smooth non-negative function with $\chi_0(t)=1$ for $|t|\leq \frac{1}{4}$ and $\chi_0(t)=0$ for $|t|\geq \frac{1}{2}$. This concludes the canonical construction of the solutions to transport equations. It is straightforward to see that the solutions $\mu_{\pm}$ are compactly supported in the set $W$ and thus we have obtained solutions to \eqref{mu_+}--\eqref{mu_-} on the entire set $U$ as well. By rewriting the functions in the original $(t,x)$-coordinates we obtain global smooth solutions to \eqref{mu_+}--\eqref{mu_-} on $M$.

Before closing this section and for the sake of future reference in Section~\ref{sec_thm_q}, we give a more detailed description of the product $$\tilde{\mu}_{+}(t_0,z_1,z'_0)\,\tilde{\mu}_{-}(t_0,z_1,z'_0),$$ with $a-\delta\leq z_1\leq b+\delta$. To this end, we first note that in view of the equations \eqref{mu_0_z} for $\tilde{\mu}_{\pm}$, there holds
\begin{equation}
\label{ODE_along_curve}
\p_{z_1}(\tilde{\mu}_+\,\tilde{\mu}_-)+  Z^\star(\gamma^{-1}\Delta\phi)\,\tilde{\mu}_+\,\tilde{\mu}_-=0.
\end{equation}
Next, recalling \eqref{phi_z} and Remark~\ref{remark_notation} we write
\begin{equation}
\label{cpm_rand_1}
\gamma^{-1}\Delta\phi=\gamma^{-1}\,\Delta_{\gamma^{-1}\g}z_1
=\gamma^{\frac{n}{2}-1}(\det g')^{-\frac{1}{2}}\, \p_{z_1}(\gamma^{1-\frac{n}{2}}\,(\det g')^{\frac{1}{2}}).
\end{equation}
Combining this with \eqref{ODE_along_curve} and the initial  condition \eqref{mu_initial} it follows that 
\begin{equation}
\label{mu_product_exp}
(\tilde{\mu}_{+}\,\tilde{\mu}_{-}\,\gamma^{1-\frac{n}{2}}\,(\det g')^{\frac{1}{2}})(t_0,z_1,z'_0)=\kappa\,
\end{equation}
for all $z_1\in [a-\delta,b+\delta]$ where $\kappa>0$ is an explicit constant depending only on $M$, $x_0$, $t_0$ and $z'_0$.
\section{Determination of $\gamma$ from $\Gamma_{\gamma,q}$ via boundary rigidity}
\label{sec_thm}

We will need the following geometric lemma.

\begin{lemma}
\label{lem_geo}
Let $\Omega\Subset \widetilde\Omega\subset \R^n$, $n\geq 2$, be bounded simply connected domains with  smooth boundaries. For $j=1,2,$ suppose that $g_j$ is a smooth Riemannian metric on $\widetilde\Omega$ with the property that $(\overline{\widetilde\Omega},g_j)$ is simple. Assume also that 
\begin{equation}
\label{g_12}
g_1=g_2 \quad \text{on $\widetilde\Omega\setminus \Omega$}.\end{equation}
Given any $x,y \in \overline{\widetilde \Omega}$, let $d_j(x,y)$ be the distance between $x$ and $y$ with respect to the metric $g_j$ and write $\zeta^{(j)}:[0,d_j(x,y)]\to \widetilde\Omega$ for the unique unit speed geodesic (w.r.t $g_j$) that connects the point $x$ to $y$. Let 
\begin{equation}\label{l_maximal}
\ell = \sup_{x\in \p\widetilde\Omega,y\in \p\Omega} \left|d_1(x,y)-d_2(x,y)\right|,\end{equation}
and suppose that $x_0\in \p\widetilde\Omega$, $y_0\in \p\Omega$ are two points where the latter maximal difference of lengths is achieved. Then, there holds
\begin{equation}\label{lem_geo_claim}
\dot{\zeta}^{(1)}(0) = \dot{\zeta}^{(2)}(0),\qquad \dot{\zeta}^{(1)}(d_1(x_0,y_0))=\dot{\zeta}^{(2)}(d_2(x_0,y_0)).\end{equation}
In other words, the two curves $\zeta^{(1)}$ and $\zeta^{(2)}$ have the same tangent vectors at both their end points $x_0$ and $y_0$.
\end{lemma}

\begin{proof}
In view of the the hypothesis of the lemma, we first note that for $j=1,2,$ $d_j:\overline{\widetilde\Omega}\times\overline{\widetilde\Omega} \to [0,\infty)$ is a continuous function which is smooth away from the diagonal elements and also that given each $x$ and $y$ on $\overline{\widetilde\Omega}$, there exists a unique unit speed geodesic (w.r.t. $g_j$) that connects $x$ to $y$. We first consider a pathological example where $\ell=0$. In this case, it follows from \eqref{g_12} that
$$d_1(x,y)=d_2(x,y) \quad \forall \, (x,y) \in (\overline{\widetilde\Omega}\setminus\Omega)\times (\overline{\widetilde\Omega}\setminus \Omega).$$
Therefore, fixing $x_0, y_0$ as in the statement of the lemma and by varying $x\in \overline{\tilde\Omega}\setminus \Omega$ in a neighborhood of the point $x_0$ and recalling \eqref{g_12}, it follows that $\nabla^{g_1} d_1(x,y_0)|_{x=x_0} =  \nabla^{g_2} d_2(x,y_0)|_{x=x_0}$. As $\nabla^{g_j}d_j(x,y_0)|_{x=x_0}$ is the unit vector tangent to $\zeta^{(j)}$ at the point $x_0$, \eqref{lem_geo_claim} follows at the point $x_0$. The claim at the other end point $y_0$ follows analogously.   

Thus we may consider the case $\ell>0$. This also removes the possibility that $\zeta^{(1)}$ or $\zeta^{(2)}$ lies completely outside $\Omega$. Hence, $\zeta^{(j)}$, $j=1,2,$ both hit $\p\Omega$ transversally at some points $t_j \in (0,d_j(x_0,y_0))$ respectively. Since $\p \Omega$ is strictly convex with respect to $g_1$ and $g_2$ it follows that 
\begin{equation}
\label{trans}
g_j(\nu_{g_j}(x_0),\dot{\zeta}^{(j)}(0))<0, \quad g_j(\nu_{g_j}(y_0),\dot{\zeta}^{(j)}(d_j(x_0,y_0)))>0,\quad j=1,2,
\end{equation}
where $\nu_{g_j}$ is the exterior normal unit vector field on $\p \Omega$ with respect to $g_j$. 

We will first prove the claim \eqref{lem_geo_claim} at the point $x_0$. Let us define the function
$$ F(x) = d_1(x,y_0)-d_2(x,y_0),$$
for $x$ in a neighborhood $O\subset \p\widetilde\Omega$ of $x_0$ . Observe that $F$ is a smooth function on $O$ and also that $x_0$ is a critical point for $F$. As $\ell>0$ and in view of \eqref{g_12} it follows that
\begin{equation}\label{d_eq_1} 
g(\nabla^{g_1} d_1(x,y_0)|_{x=x_0}-\nabla^{g_2}d_2(x,y_0)|_{x=x_0},w)=0 \quad \forall w \in T_{x_0}O,\end{equation}
where $g:=g_1|_{\p\widetilde\Omega}=g_2|_{\p\widetilde\Omega}$.  On the other hand, by definition,
$$ |\nabla^{g_1}d_1(x,y_0)|_{g_1(x)}=|\nabla^{g_2}d_2(x,y_0))|_{g_2(x)}=1,$$
for all $x\in \overline{\widetilde\Omega}$ away from $y_0$. Combining this identity with \eqref{d_eq_1} and \eqref{trans}, we conclude that 
$$\nabla^{g_1} d_1(x,y_0)|_{x=x_0} =  \nabla^{g_2} d_2(x,y_0)|_{x=x_0}.$$ 
This yields the claim \eqref{lem_geo_claim} at the point $x_0$. The claim at the point $y_0$ follows analogously. 
\end{proof}

We are ready to prove the unique recovery for $\gamma$ from $\Gamma_{\gamma,q}$.

\begin{proof}[Proof of $\gamma_1=\gamma_2$ in Theorem~\ref{thm_alt}]
Let $\widetilde\Omega$ be an infinitesimal extension of the domain $\Omega$ into a larger domain with a smooth boundary, so that given each $t\in [0,T]$, the Riemannian manifold $(\widetilde\Omega, \g_j(t,\cdot))$, $j=1,2,$ is also simple.  Let $t_0\in (0,T)$ and define $d_j(x,y)$ with $(x,y)\in \overline{\widetilde\Omega} \times \overline{\widetilde\Omega}$  to be the Riemannian distance from the point $x$ to the point $y$ with respect to the metric $\g_j(t_0,\cdot)$. Let  
$$\ell_{t_0} = \sup_{x\in \p\widetilde\Omega,y\in \p\Omega} \left|d_1(x,y)-d_2(x,y)\right|.$$
We claim that $\ell_{t_0}=0$. To this end, we give a proof by contradiction and suppose for now that $\ell_{t_0}>0$. Let $x_0 \in \p\widetilde\Omega$ and $y_0\in \p\Omega$ be chosen such that the latter maximal value is attained. We will assume without loss of generality that $d_2(x_0,y_0)>d_1(x_0,y_0)$ and write
\begin{equation}\label{l_dif}
\ell_{t_0} = d_2(x_0,y_0)-d_1(x_0,y_0)>0.\end{equation}
Let us define $\zeta^{(1)}_{t_0,x_0}:[0,d_1(x_0,y_0)]\to \R^{n}$ and $\zeta^{(2)}_{t_0,x_0}:[0,d_2(x_0,y_0)]\to \R^{n}$ to be the corresponding unit speed geodesics that connect the point $x_0$ to $y_0$ with respect to the Riemannian metric $\g_1(t_0,\cdot)$ and $\g_2(t_0,\cdot)$ respectively. Applying Lemma~\ref{lem_geo} with $g_j(\cdot)=\g_j(t_0,\cdot)$ we note that the unit speed geodesics $\zeta_{t_0,x_0}^{(1)}$ and $\zeta_{t_0,x_0}^{(2)}$ have identical tangent vectors at the two end points.  Moreover, as both geodesics hit $\p \Omega$ transversally at some point between $x_0$ and $y_0$ and as $\g_1(t_0,\cdot)=\g_2(t_0,\cdot)$ on $\widetilde\Omega\setminus\Omega$, we conclude that 
\begin{itemize}
\item[(P)]{ There exists a point $p=\zeta_{t_0,x_0}^{(1)}(s_0)=\zeta_{t_0,x_0}^{(2)}(s_0)\in \p\Omega$ and such that the two geodesics $\zeta_{t_0,x_0}^{(j)}(s)$ are identical for all $s\in [0,s_0]$.}
\end{itemize}
For $\tau>0$ sufficiently large, we define $u^{(1)}_{+,\tau}$ to be the canonical exponentially growing solution to \eqref{heat_alt} (with $\gamma=\gamma_1$ and $q=q_1$) as in Section~\ref{sec_exp} constructed with respect to the parameters $t_0$, $x_0$, $\zeta_{t_0,x_0}^{(1)}$ and a fixed $\delta>0$ sufficiently small, independent of $\tau>0$. We also define $u^{(2)}_{-,\tau}$ to be the canonical exponentially decaying solution to \eqref{heat_adjoint} (with $\gamma=\gamma_2$ and $q=q_2$) as in Section~\ref{sec_exp} constructed with respect to the parameters $t_0$, $x_0$, $\zeta_{t_0,x_0}^{(2)}$ and $\delta>0$. Recall that 
\begin{equation}
\label{u_exp_pm_1}
u^{(1)}_{+,\tau}(t,x)= e^{\tau^2 t+\tau \phi^{(1)}(t,x)}\left(\mu_{+}^{(1)}(t,x)+R^{(1)}_{+,\tau}(t,x)\right),
\end{equation}
and
\begin{equation}
\label{u_exp_pm_2}
u^{(2)}_{-,\tau}(t,x)= e^{-\tau^2 t-\tau \phi^{(2)}(t,x)}\left(\mu_{-}^{(2)}(t,x)+R^{(2)}_{-,\tau}(t,x)\right),
\end{equation}
Here, and in what follows, the superscripts $(j)$, $j=1,2$, depict that the respective constructions in Section~\ref{sec_exp} are carried out with $q=q_j$ and $\gamma=\gamma_j$. We note that the phase functions $\phi^{(j)}$ are constructed as in Section~\ref{sec_phi} with respect to $\gamma_j$ and that the amplitudes $\mu^{(j)}_{\pm}$ are as canonically constructed in Section~\ref{sec_mu} in their respective normal polar coordinates. We emphasize that the normal polar coordinates for each construction is different as the choice of the normal coordinates also depends on $\gamma_j$. To avoid confusion we will work in the same original $(t,x)$-coordinate system for both solutions. Finally, recall that the remainder terms satisfy the decay properties \eqref{remainder_estimates}. Next, let us define
\begin{equation}
\label{boundary_dich_choices}
f = u_{+,\tau}^{(1)}|_{\Sigma} \quad \text{and}\quad h=u_{-,\tau}^{(2)}|_{\Sigma},
\end{equation}
and observe that since $\delta>0$ is assumed to be sufficiently small, there holds $f \in \mathcal H(\Sigma)$ while $h\in \mathcal H^\star(\Sigma)$. Appling Lemma~\ref{boundary_int} together with the fact that $\Gamma_{\gamma_1,q_1}=\Gamma_{\gamma_2,q_2}$, it follows that 
$$
\int_{\Sigma}\Gamma_{\gamma_1,q_1}(f)\,h\,dV_{\Sigma} =\int_{\Sigma}f\,\Gamma^\star_{\gamma_2,q_2}(h)\,dV_{\Sigma}.
$$
This can now be rewritten as follows.
\begin{multline*}
\int_\Sigma \p_\nu(e^{\tau\,\phi^{(1)}}(\mu^{(1)}_++R_{+,\tau}^{(1)}))\,e^{-\tau\,\phi^{(2)}}(\mu^{(2)}_-+R_{-,\tau}^{(2)})\,dV_{\Sigma}\\
-\int_{\Sigma} e^{\tau\,\phi^{(1)}}(\mu^{(1)}_++R_{+,\tau}^{(1)})\,\p_\nu(e^{-\tau\,\phi^{(2)}}(\mu^{(2)}_-+R_{-,\tau}^{(2)}))\,dV_\Sigma=0.  
\end{multline*}
As $R_{+,\tau}^{(1)}|_{\Sigma}=R_{-,\tau}^{(2)}|_{\Sigma}=0$, we may reduce the expression as follows
\begin{multline}
\label{big_asymp_exp}
\textrm{I}=\int_\Sigma \p_\nu(e^{\tau\,\phi^{(1)}}\mu^{(1)}_+)\,e^{-\tau\,\phi^{(2)}}\mu^{(2)}_-dV_{\Sigma}\\
+\int_\Sigma \p_\nu(e^{\tau\,\phi^{(1)}}R_{+,\tau}^{(1)})\,e^{-\tau\,\phi^{(2)}}\mu^{(2)}_-\,dV_{\Sigma}-\int_\Sigma\p_\nu(e^{-\tau\,\phi^{(2)}}\mu^{(2)}_-)\,e^{\tau\,\phi^{(1)}}\mu^{(1)}_+dV_{\Sigma}\\
-\int_\Sigma \p_\nu(e^{-\tau\,\phi^{(2)}}R_{-,\tau}^{(2)})\,e^{\tau\,\phi^{(1)}}\mu^{(1)}_+\,dV_{\Sigma}=0.
\end{multline}
Let us consider the two points $(t_0,p) \in \Sigma$ (see the property (P)) and $(t_0,y_0)\in \Sigma$. Owing to the definition of $\mu_{+}^{(1)}$ and $\mu_-^{(2)}$ given by \eqref{mu_0_def}, it is clear that each of the four integrands in \eqref{big_asymp_exp} are supported in an open neighborhood $\mathcal U$ of the point $(t_0,p)$ on $\Sigma$ and a neighborhood $\mathcal W$ of the point $(t_0,y_0)\in \Sigma$. Furthermore, the size of these neighborhoods depends on $\delta$. Thus, we may break the integrations in \textrm{I} in \eqref{big_asymp_exp} over the two sets $\mathcal W$ and $\mathcal U$, calling them $\textrm{II}$ and $\textrm{III}$ respectively. In particular, in view of \eqref{l_dif} and the definition of the phase functions $\phi^{(j)}$, $j=1,2,$ we can assume (by taking $\delta>0$ sufficiently small) that there holds:
\begin{equation}
\label{difference_W}
\phi^{(2)}(t,x)-\phi^{(1)}(t,x)>\frac{1}{2}\,\ell_{t_0},\quad \forall\, (t,x)\in \mathcal W.
\end{equation}
Let us first consider the integrations in \eqref{big_asymp_exp} over the set $\mathcal W$. Over this set, it is straightforward (using the remainder estimates \eqref{remainder_estimates}) to obtain 
\begin{multline}
\label{big_asymp_exp_W}
|\textrm{II}|=|\int_{\mathcal W} \p_\nu(e^{\tau\,\phi^{(1)}}\mu^{(1)}_+)\,e^{-\tau\,\phi^{(2)}}\mu^{(2)}_-dV_{\mathcal W}\\
+\int_{\mathcal W} \p_\nu(e^{\tau\,\phi^{(1)}}R_{+,\tau}^{(1)})\,e^{-\tau\,\phi^{(2)}}\mu^{(2)}_-\,dV_{\mathcal W}-\int_{\mathcal W}\p_\nu(e^{-\tau\,\phi^{(2)}}\mu^{(2)}_-)\,e^{\tau\,\phi^{(1)}}\mu^{(1)}_+dV_{\mathcal W}\\
-\int_{\mathcal W} \p_\nu(e^{-\tau\,\phi^{(2)}}R_{-,\tau}^{(2)})\,e^{\tau\,\phi^{(1)}}\mu^{(1)}_+\,dV_{\mathcal W}|\leq C_0\,\tau\,e^{-\frac{\ell_{t_0}}{2}\,\tau},
\end{multline}
for all $\tau>0$ sufficiently large and some constant $C_0>0$ independent of $\tau$. On the other hand, over the set $\mathcal U$, we can give a more detailed analysis of the four integrals in \eqref{big_asymp_exp}. Indeed, the advantage over the set $\mathcal U$ is that in view the property (P)
that was established earlier, there holds
$$\phi(t,x):=\phi^{(1)}(t,x)=\phi^{(2)}(t,x) \quad \forall\, (t,x)\in \mathcal U.$$ 
Therefore, by applying the latter equality together with the remainder estimates \eqref{remainder_estimates} it follows that
 \begin{multline*}
 \textrm{III}=\int_{\mathcal U} \p_\nu(e^{\tau\,\phi}\mu^{(1)}_+)\,e^{-\tau\,\phi}\mu^{(2)}_-dV_{\mathcal U}\\
 +\int_{\mathcal U} \p_\nu(e^{\tau\,\phi}R_{+,\tau}^{(1)})\,e^{-\tau\,\phi}\mu^{(2)}_-\,dV_{\mathcal U}-\int_{\mathcal U}\p_\nu(e^{-\tau\,\phi}\mu^{(2)}_-)\,e^{\tau\,\phi}\mu^{(1)}_+dV_{\mathcal U}\\
 -\int_{\mathcal U} \p_\nu(e^{-\tau\,\phi}R_{-,\tau}^{(2)})\,e^{\tau\,\phi}\mu^{(1)}_+\,dV_{\mathcal U}\\
 = 2\,\tau\,\int_{\mathcal U}(\p_\nu\phi)\,\mu_+^{(1)}\,\mu_-^{(2)}\,dV_{\mathcal U} + O(\sqrt{\tau}),
 \end{multline*}
 where $O(\sqrt{\tau})$ denotes a term that is bounded in absolute value by a constant times $\sqrt{\tau}$. As $\mu_{+}^{(1)}$ and $\mu_{-}^{(2)}$ are strictly positive at $(t_0,p)\in \Sigma$ and as they are both non-negative on $\mathcal U$ (see \eqref{mu_0_def} and also \eqref{mu_product_exp}) and finally as $\p_\nu\phi$ is strictly negative on $\overline{\mathcal U}$ (see \eqref{phi_z}), it follows that
 \begin{equation}
  \label{big_asymp_exp_U}
  |\textrm{III}|\geq  C_1\,\tau,
 \end{equation}
 for all $\tau$ sufficiently large and some constant $C_1>0$ independent of $\tau$. By combining the fact that $0=\textrm{I}=\textrm{II}+\textrm{III}$ with \eqref{big_asymp_exp_W} and \eqref{big_asymp_exp_U}, it follows that
  $$C_1\,\tau \leq |\textrm{III}|\leq |\textrm{II}| \leq C_0\,\tau\,e^{-\frac{\ell_{t_0}}{2}\,\tau},$$
  for all sufficiently large $\tau>0$. This gives us a contradiction as $\ell_{t_0}$ was assumed to be positive. Therefore, we conclude that $\ell_{t_0}=0$. Thus, in view of the definition of $\ell_{t_0}$, we obtain
  $$ d_1(x,y)=d_2(x,y)\quad \forall\, (x,y)\in \p\Omega\times\p\Omega.$$
  Therefore, writing $\mathbb E^n$ for the Euclidean metric on $\R^n$, we have shown that the two Riemannian manifolds $(\overline\Omega, \gamma_1(t_0,x)\mathbb E^n )$ and $(\overline\Omega, \gamma_1(t_0,x)\mathbb E^n )$ have the same boundary distance data. By applying the boundary rigidity result
  \cite{Mu1} in the case $n=2$ for simple conformally Euclidean Riemannian manifolds and by applying the boundary rigidity result \cite{Mu2} in the case $n\geq 3$ for conformally Euclidean simple manifolds, we conclude that
  $$\gamma_1(t_0,x)=\gamma_2(t_0,x)\quad \forall\, x \in\Omega.$$
  The result now follows since $t_0\in(0,T)$ is arbitrary. 
\end{proof}

\section{Determination of $q$ from $\Gamma_{\gamma,q}$ via geodesic ray transform}
\label{sec_thm_q}
To complete the proof of Theorem~\ref{thm_alt}, it remains to show that $q_1=q_2$ on $M$. Let us mention that as $q$ is a lower order coefficient, its recovery is analogous to previous results on recovery of lower order coefficients, see e.g. \cite{Isa92} that recovers $q$ when $\gamma$ is identically a constant. We will assume throughout this section that the hypothesis of Theorem~\ref{thm_alt} is satisfied and note in view of the last section that there holds:
\begin{equation}
\label{gamma_eq}
\gamma(t,x):=\gamma_1(t,x)=\gamma_2(t,x),\quad \forall\,(t,x)\in M.
\end{equation} 
We have the following lemma.
\begin{lemma}
	\label{lem_prod_iden}
	Assume that the hypotheses of Theorem~\ref{thm_alt} is satisfied and in view of Section~\ref{sec_thm} let $\gamma$ be as defined by \eqref{gamma_eq}. Let $u$ be a solution to \eqref{heat_alt} with $q=q_1$ for some Dirichlet boundary data $f\in \mathcal H(\Sigma)$. Let $w$ be a solution to \eqref{heat_adjoint} with $q=q_2$ for some Dirichlet boundary data $h\in \mathcal H^\star(\Sigma)$.  There holds,
	\begin{equation}
	\label{q_dif_prod}
	\int_M (q_2(t,x)-q_1(t,x))\,u(t,x)\,w(t,x)\,dt\,dx=0.\end{equation}
\end{lemma}
\begin{proof}
	Let $v$ be the unique solution to \eqref{heat_alt} with $q=q_2$ and Dirichlet boundary data $f$. The function 
	$$\tilde u = u-v$$ 
	satisfies the following boundary value problem
	\begin{equation}\label{heat_alt_1}
	\begin{aligned}
	\begin{cases}
	\gamma(t,x)\,\p_t \tilde{u}-\Delta \tilde u+q_2\, \tilde u=(q_2-q_1)\,u 
	&\text{on  $M^{\textrm{int}}$},
	\\
	\tilde u= 0 & \text{on $\Sigma$},
	\\
	\tilde u|_{t=0}= 0 & \text{on $\Omega$}.
	\end{cases}
	\end{aligned}
	\end{equation}
Moreover, in view of the equality $\Gamma_{\gamma,q_1}=\Gamma_{\gamma,q_2}$, there holds:
$$ \p_\nu \tilde u|_{\Sigma}=0.$$
The claim follows from multiplying the first equation in \eqref{heat_alt_1} by $w$ and integrating over $M$ by parts. 
\end{proof}
We are ready to complete the proof of Theorem~\ref{thm_alt}.
\begin{proof}[Proof of $q_1=q_2$ in Theorem~\ref{thm_alt}]
Let $t_0\in (0,T)$ and let $x_0 \in \R^n\setminus \overline\Omega$ be sufficiently close to $\p \Omega$. Let us define the Riemannian metric $\g(t,x)$ on $[0,T]\times \R^n$ via
$$\g(t,x)=\gamma(t,x)\,\left((dx_1)^2+\ldots+(dx_n)^2\right).$$
Next, let $y_0\in \p\Omega$ be another point on $\p \Omega$ such that the unique unit speed geodesic $\zeta_{t_0,x_0}:[0,d]\to \R^n$ (with respect to the metric $\g(t_0,\cdot)$) that connects $x_0$ to $y_0$ also hits the boundary $\p\Omega$ transversally at some point $s\in (0,d)$. For $\tau>0$ sufficiently large, we consider two canonical families of exponentially growing and decaying solutions as follows. Let $u^{(1)}_{+,\tau}$ be the canonical exponentially growing solution to \eqref{heat_alt} (with $q=q_1$) as in Section~\ref{sec_exp} constructed with respect to the parameters $t_0$, $x_0$, $\zeta_{t_0,x_0}$ and $\delta>0$, while $u^{(2)}_{-,\tau}$ is the canonical exponentially decaying solution to \eqref{heat_adjoint} (with $q=q_2$) as in Section~\ref{sec_exp} constructed with respect to the same parameters $t_0$, $x_0$, $\zeta_{t_0,x_0}$ and $\delta>0$. Recall that 
\begin{equation}
\label{u_exp_pm_3}
u^{(1)}_{+,\tau}(t,x)= e^{\tau^2 t+\tau \phi(t,x)}\left(\mu_{+}(t,x)+R^{(1)}_{+,\tau}(t,x)\right),
\end{equation}
and
\begin{equation}
\label{u_exp_pm_4}
u^{(2)}_{-,\tau}(t,x)= e^{-\tau^2 t-\tau \phi(t,x)}\left(\mu_{-}(t,x)+R^{(2)}_{-,\tau}(t,x)\right),
\end{equation}
where we are using the fact that the phase function $\phi$ and the amplitude functions $\mu_{\pm}$ are independent of the functions $q_1$ and $q_2$. Observe also that in view of \eqref{remainder_estimates}, there holds
\begin{equation}
\label{remainder_est_final}
\|R_{+,\tau}^{(1)}\|_{L^2(M)}+\|R_{-,\tau}^{(2)}\|_{L^2(M)}\leq C\,\tau^{-1},
\end{equation}
for some constant $C>0$ independent of $\tau$.

We apply Lemma~\ref{lem_prod_iden} with $u=u_{+,\tau}^{(1)}$ and $w=u_{-,\tau}^{(2)}$. Therefore,
$$\int_M (q_2(t,x)-q_1(t,x))\,u_{+,\tau}^{(1)}(t,x)\,u_{-,\tau}^{(2)}(t,x)\,dt\,dx=0.$$
Next, by using the identities \eqref{u_exp_pm_3}--\eqref{u_exp_pm_4} and estimate \eqref{remainder_est_final}, and by taking the limit as $\tau\to \infty$ we deduce that
$$\int_M (q_2(t,x)-q_1(t,x))\,\mu_{+}(t,x)\,\mu_{-}(t,x)\,dt\,dx=0.$$
Observe that the integrand above is supported in a small tubular neighborhood of the ray $\zeta_{t_0,x_0}$ depending on $\delta$. Observe also that $q_1-q_2=0$ outside $M$. Therefore, we may use the $(t,z)$-coordinates in Section~\ref{sec_mu} and rewrite the above expression as follows:
$$\int_{V} (q_1-q_2)\, \tilde{\mu}_+\,\tilde{\mu}_-\, \chi_\delta^2\, \,dt\,dV_{\gamma^{-1}(t,\cdot)\g(t,\cdot)}=0.$$
Recalling the definition \eqref{chi_def} and taking the limit as $\delta\to 0$, it follows that 
$$\int_{a-\delta}^{b+\delta} Q(z_1)\, (\tilde{\mu}_+\,\tilde{\mu}_-\,\gamma^{1-\frac{n}{2}}(\det g')^{\frac{1}{2}})(t_0,z_1,z'_0)\,dz_1=0,$$
where 
$$Q(z_1):= \gamma^{-1}(t_0,z_1,z'_0)\, (q_1-q_2)(t_0,z_1,z'_0).$$
Recalling the expression \eqref{mu_product_exp}, we conclude that
$$\int_{a-\delta}^{b+\delta} Q(z_1)\,dz_1=0.$$
The latter expression implies that the function $\gamma^{-1}(t_0,x)(q_1-q_2)(t_0,x)$ integrated over the curve $\zeta_{t_0,x_0}$ yields zero. By varying the points $x_0$ and $y_0$ we conclude that the geodesic ray transform associated to the function $\gamma^{-1}(t_0,\cdot)(q_1-q_2)(t_0,\cdot)$ on the Riemannian manifold $(\overline\Omega,\g(t_0,\cdot))$ is identically zero. As the geodesic ray transform is injective on conformally Euclidean simple manifolds, see e.g. \cite{Sh}, we conclude that
$$ q_1(t_0,x)=q_2(t_0,x)\quad \forall\, x\in \Omega.$$
The result now follows since $t_0\in (0,T)$ is arbitrary.
\end{proof}

\end{document}